\documentclass[11pt]{article}
\usepackage{amsmath, amsfonts, amsthm, amssymb, multicol}
\usepackage{graphicx}
\usepackage{float}
\usepackage{verbatim}

\allowdisplaybreaks

\usepackage[square,sort,comma,numbers]{natbib}
\usepackage{fancyhdr}
 
\numberwithin{equation}{section}

\newtheorem{thm}{Theorem}[section]
\newtheorem{lma}[thm]{Lemma}
\newtheorem{cor}[thm]{Corollary}

\newtheorem{prop}[thm]{Proposition}

\renewcommand{\epsilon}{\varepsilon}
\newcommand{\eps}{\varepsilon}

\renewcommand{\geq}{\geqslant}
\renewcommand{\leq}{\leqslant}

\newcommand{\uld}{\overline{\dim}_{\textup{loc}}}
\newcommand{\lld}{\underline{\dim}_{\textup{loc}}}

\newcommand{\ad}{\dim_{\mathrm{A}} }
\newcommand{\ld}{\dim_{\mathrm{L}}  }
\newcommand{\hd}{\dim_{\mathrm{H}}  }
\newcommand{\bd}{\dim_{\mathrm{B}}  }
\newcommand{\urd}{\overline{\dim}_{\mathrm{reg}} \, }
\newcommand{\lrd}{\underline{\dim}_{\mathrm{reg}} \, }

\newcommand{\hm}{d_\mathbb{H}}

\newcommand{\ps}{\mu_{ \mathrm{PS}}}
\newcommand{\ls}{L(\Gamma)}

\newcommand{\isom}{\mathrm{Con}(d)}

\newcommand{\kmin}{k_{\min}}
\newcommand{\kmax}{k_{\max}}

\title{ \vspace{-20mm}Regularity of Kleinian limit sets   and  \\ Patterson-Sullivan measures}

\author{Jonathan M. Fraser}

\begin{document}


\maketitle

\begin{abstract}
We consider several (related) notions of geometric regularity in the context of limit sets of geometrically finite Kleinian groups and associated  Patterson-Sullivan measures.  We begin by computing the upper and lower regularity dimensions of the Patterson-Sullivan measure, which involves controlling the relative measure of concentric balls.  We then compute the Assouad and lower dimensions of the limit set, which involves controlling local doubling properties.  Unlike the Hausdorff, packing, and box-counting dimensions, we show that the Assouad and lower dimensions are not necessarily given by the Poincar\'e exponent.
\end{abstract}

\section{Introduction}

\subsection{Limit sets of Kleinian groups and the Patterson-Sullivan measure}

 We consider fractal sets and measures arising from  discrete groups of isometries acting on hyperbolic space.  For integer $d \geq 1$, we model $(d+1)$-dimensional hyperbolic space using the Poincar\'e ball
\[
\mathbb{D}^{d+1} = \left\{ z \in \mathbb{R}^{d+1} \ : \ |z| < 1 \right\}
\]
equipped with the hyperbolic metric $\hm$ defined by
\[
d s =  \frac{2|dz|}{1-|z|^2}.
\]
The \emph{boundary at infinity} of the space $(\mathbb{D}^{d+1}, \hm)$ is $\mathbb{S}^{d} = \left\{ z \in \mathbb{R}^{d+1} \ : \ |z| = 1 \right\}$ and the group of  isometries is given by the stabliser of $\mathbb{D}^{d+1}$ in the M\"obius group acting on $\overline{\mathbb{R}^{d+1}}$.  We denote the group of orientation preserving isometries of $(\mathbb{D}^{d+1}, \hm)$ by $\isom$, and note that it is isomorphic to the (orientation preserving) M\"obius group acting on $\overline{\mathbb{R}^{d}}$.  We will sometimes appeal to the upper half-space model of hyperbolic space, where $\mathbb{D}^{d+1}$ is replaced by $\mathbb{H}^{d+1} = \mathbb{R}^d \times (0,\infty)$ equipped with the analogous metric, but this is purely for aesthetic reasons as these two models of hyperbolic space are of course isometric and, moreover, there is a M\"obius transformation between the corresponding boundaries which (we will see) preserves all of our notions of dimension.  We refer the reader to \cite{anderson, beardon, maskit, kapovich} for a more detailed study of hyperbolic geometry, including the isometry group, and the correspondence between, and equivalence of, the two models we use.

A \emph{Kleinian group} is a discrete  subgroup of $\isom$ and such groups act properly discontinuously on $\mathbb{D}^{d+1}$ but may fail to act discontinuously on the boundary.  The \emph{limit set} of a Kleinian group $\Gamma$ is the set of points where the action fails to be discontinuous and it carries a lot of geometric information concerning the group.  More precisely, writing $ \mathbf{0} = (0, \dots, 0) \in \mathbb{D}^{d+1}$, the limit set is defined by
\[
L(\Gamma) = \overline{\Gamma( \mathbf{0})} \setminus \Gamma( \mathbf{0}).
\]
This is a compact subset of $\mathbb{S}^{d}$ and often has a beautiful and subtle fractal structure.  Note that for definiteness we metrise $\mathbb{S}^{d}$ with the Euclidean metric $\| \cdot \|$ inherited from $\mathbb{R}^{d+1}$, although the standard Riemannian metric on $\mathbb{S}^d$ is bi-Lipschitz equivalent to this metric and so from a dimension point of view these two natural metrics on the limit set are equivalent.

If the limit set is empty or consists only of one or two points, then the Kleinian group is called \emph{elementary} and otherwise it is \emph{non-elementary}, in which case  the limit set is necessarily uncountable.  A Kleinian group is called \emph{geometrically finite} if it has a fundamental domain with finitely many sides.  The Poincar\'e exponent of a Kleinian group $\Gamma$ is defined by
\[
\delta(\Gamma) = \inf \left\{ s >0 \ : \ \sum_{g \in \Gamma} \exp(-s \, \hm(\mathbf{0},g( \mathbf{0})))  < \infty \right\}
\]
and plays a central role in the geometry and dimension theory of $\Gamma$.   In particular, the limit set of a non-elementary geometrically finite Kleinian group has Hausdorff dimension equal to $\delta(\Gamma)$. This  important result goes back to the influential papers of Patterson (for Fuchsian groups with some assumptions on parabolic elements) \cite[Theorems 4.1 and 5.1]{patterson} and Sullivan (for the general higher dimensional case) \cite[Theorem 1]{sullivan}.  Almost 20 years later it was  shown that in this setting the packing and box-counting dimensions of the limit set are also given by $\delta(\Gamma)$.   This result is due independently to Bishop and Jones \cite[Corollary 1.5]{bishopjones} and Stratmann and Urba\'nski \cite[Theorem 3]{stratmannurbanski}.  For a review of the Hausdorff, box-counting, and packing dimensions, see \cite[Chapters 2 and 3]{falconer}.  When discussing geometrically finite groups, we will only mention the Hausdorff dimension, which we denote by $\hd$, since the Hausdorff, packing, and box-counting dimensions coincide in this case.

Limit sets of non-elementary geometrically finite Kleinian groups are also known to carry an atomless  conformal ergodic Borel probability measure $\ps$ of Hausdorff dimension $\delta(\Gamma)$, known as the \emph{Patterson-Sullivan measure}.  Again, we will only discuss the (lower) \emph{Hausdorff} dimension of the Patterson-Sullivan measure, but this is known to equal the upper packing dimension (indeed, the Patterson-Sullivan measure is exact dimensional, see \cite{stratmannvelani}).  See \cite[Chapter 10]{techniques} for a review of the dimensions of measures.

The Patterson-Sullivan measure has played a central role in the geometry of Kleinian groups and, along with the limit set itself, is one of the key objects we study.   Stratmann and Velani's \emph{global measure formula} gives a formula for the measure of any ball up to uniform constants and will be particularly relevant to our work, see \cite[Theorem 2]{stratmannvelani}.  Before stating this formula we need to introduce some more notation, particularly concerned with parabolic elements, that is, elements in $\isom$ with precisely one fixed point in $\mathbb{S}^{d}$.

Fix a non-elementary  geometrically finite Kleinian group and suppose $\Gamma$ is not \emph{parabolic free}, that is, it contains at least one parabolic element.   Let $P \subseteq \ls$ denote the countable set of all parabolic fixed points, that is, points fixed by parabolic elements of $\Gamma$.  We may fix a standard set of pairwise disjoint  horoballs $\{H_p\}_{p \in P}$, where each $H_p$ is a horoball with base point $p$, that is, a closed Euclidean ball whose interior lies inside $\mathbb{D}^{d+1}$  and is tangent to $\mathbb{S}^{d}$ at $p$.  Moreover, the horoballs can be chosen such that $g(H_p) = H_{g(p)}$ for all $g \in \Gamma$ and $p \in P$.  Thus, although the choice of standard  horoballs is not unique, any given choice reflects the geometry of the limit set in a representative way.  The stabiliser of a parabolic fixed point $p$ cannot contain hyperbolic or loxodromic elements since if a subgroup of $\isom$ contains a parabolic and a hyperbolic/loxodromic element which fix the same point then the group is not discrete. Therefore the parabolic elements in the stabiliser of $p$ in $\Gamma$ generate a free Abelian group of finite index (as a subgroup of the stabiliser).  We define $k(p)$ to be the maximal rank of a free Abelian subgroup of the stabiliser of $p$ in $\Gamma$, which is necessarily generated by $k(p)$ parabolic elements all fixing $p$.  For an account of standard horoballs and ranks of parabolic elements, we refer the reader to the opening discussion in \cite{stratmannurbanski}.  We note the important fact that $\delta(\Gamma)> k(p)/2$ for all $p \in P$. 

 Given $z \in L(\Gamma)$ and $t>0$, let $z_t \in \mathbb{D}^{d+1}$ be the unique point on the geodesic ray joining $ \mathbf{0}$ to $z$ which is at hyperbolic distance $t$ from $ \mathbf{0}$.  Write  $S(z,t) \subset \mathbb{S}^{d}$ to denote the \emph{shadow at infinity} of the $d$-dimensional (hyperbolic) hyperplane passing through $z_t$ normal to the geodesic ray joining $ \mathbf{0}$ to $z$.  Basic hyperbolic geometry shows that  $S(z,t)$ is a Euclidean ball centred at $z$ with radius uniformly comparable to $e^{-t}$.  The global measure formula states that there is a uniform constant $C>1$ such that for all $z \in \ls$ and all $t>0$ we have
\begin{equation} \label{global}
\frac{1}{C} \ \leq  \ \frac{\ps(S(z,t))}{\exp(-t \delta(\Gamma)- \rho(z,t) (\delta(\Gamma) - k(z,t)))} \  \leq \  C
\end{equation}
where $k(z,t) = k(p)$ if $z_t \in H_p$ for some $p$ and $0$ otherwise and
\[
\rho(z,t) = \inf \{ \hm(z_t, y) \ : \ y \notin H_p \}
\]
if $z_t \in H_p$ for some $p$ and $0$ otherwise.  Note that if we choose a different set of standard horoballs, then the constant $C$ can change and so for definiteness we fix a set of standard horoballs, and therefore a constant $C$, for the rest of the paper.  The global measure formula still holds if $\Gamma$ is parabolic free and in that case it simplifies to
\begin{equation} \label{global2}
\frac{1}{C} \ \leq  \ \frac{\ps(S(z,t))}{\exp(-t \delta(\Gamma))} \  \leq \  C.
\end{equation}

\subsection{Regularity dimensions of measures and Assouad dimensions of sets}

In this section we work with a general complete metric measure space $(X,d, \mu)$ but our results will mostly concern the space $(\ls, \| \cdot \|, \ps)$.  

The upper and lower regularity dimensions of $\mu$ describe the optimal global control on the relative measure of concentric balls.  These dimensions were introduced formally in \cite{anti1,anti2} motivated by previous work on the existence of doubling measures, see for example \cite{luksak, konyagin}.   For some basic properties and the explicit computation of the (upper) regularity dimension in some particular contexts, see \cite{fraserhowroyd}.  The upper and lower regularity dimensions provide a generally applicable refinement of the notion of Ahlfors-David regularity.  Recall that a (non-atomic) measure is $s$-Ahlfors-David regular if the ratio $\mu(B(x,R))/R^s$ is uniformly bounded away from 0 and $+\infty$ for all $0<R\leq |\text{supp} (\mu)| $.  We write $\text{supp}(\mu) \subseteq X$ for the support of $\mu$,   $|E|$ for the diameter of non-empty (possibly unbounded) $E \subseteq X$, and $B(x,R)$ for the open ball of radius $R>0$ and centre $x \in X$.  The \emph{upper regularity dimension} of $\mu$ is defined by 
\begin{multline*} 
\urd \mu = \inf \Bigg\{ s \geq 0 \, \,  : \,  \text{ there exists }C  > 0\text{  such that, for all  $0< r< R< |\text{supp} (\mu)|$} \\   \text{  and all $x \in \text{supp} (\mu)$, we have }  \ \  \frac{\mu(B(x,R))}{\mu(B(x,r))} \leq C\left(\frac{R}{r}\right)^{s} \Bigg\}
\end{multline*}
and, provided $|\text{supp} (\mu)| >0$, the \emph{lower regularity dimension} of $\mu$ is defined by 
\begin{multline*} 
\lrd \mu = \sup \Bigg\{ s \geq 0 \, \,  : \,  \text{ there exists }C  > 0\text{  such that, for all  $0< r< R < |\text{supp} (\mu)| $} \\ \text{  and all $x \in \text{supp} (\mu)$, we have }  \ \  \frac{\mu(B(x,R))}{\mu(B(x,r))} \geq C\left(\frac{R}{r}\right)^{s} \Bigg\}
\end{multline*}
and otherwise it is 0.  We adopt the convention that $\inf \emptyset = + \infty$.   A measure $\mu$ is doubling if and only if $\urd \mu  <\infty$, see \cite{many, fraserhowroyd}. Also note that if a set carries an $s$-Ahlfors-David regular measure, then the upper and lower regularity dimensions coincide and equal $s$.   The regularity dimensions are heavily related to the Assouad and lower dimensions, which are purely metric notions describing the extremal scaling behaviour of a set in a metric space.   These dimensions have fundamental applications in embedding theory and quasi-conformal geometry, for example, and have recently been enjoying a period of intense interest in the fractal geometry literature.  We recall the definitions of the Assouad and lower  dimensions here, but refer the reader to \cite{Robinson, Fraser, Luukkainen, mackaytyson} for more details.  For non-empty $E \subseteq X$ and $r>0$, let $N_r (E)$ be the smallest number of open sets with diameter less than or equal to $r$ required to cover $E$.  The \emph{Assouad dimension} of a non-empty set $F\subseteq X$ is defined by
\begin{multline*} 
\ad F =  \inf \Bigg\{ s \geq 0 \, \,  : \,  \text{ there exists  }C  > 0\text{  such that, for all  $0< r< R<|F| $} \\ \text{and all $x \in F$, we have }  \ \  N_r\big( B(x,R) \cap F \big) \leq C\left(\frac{R}{r}\right)^{s} \Bigg\}
\end{multline*}
and, provided $|F| >0$, the \emph{lower dimension} of $F$ is defined by
\begin{multline*} 
\ld F = \sup \Bigg\{ s \geq 0 \, \,  : \,  \text{ there exists }C  > 0\text{  such that, for all  $0< r< R <|F| $} \\ \text{and all $x \in F$, we have }  \ \  N_r\big( B(x,R) \cap F \big) \geq C\left(\frac{R}{r}\right)^{s} \Bigg\}
\end{multline*}
and otherwise it is 0.  It is well-known that for closed $F$ we have
\[
\ld F \leq \hd F \leq \ad F.
\]
The  regularity dimensions can be thought of as the Assouad and lower dimensions of a measure.  Indeed, for any Borel probability measure $\nu$ fully supported on a closed set $F \subseteq X$, it is easy to see that $\lrd \nu \leq \ld F \leq \ad F \leq \urd \nu$ but a deeper fact is that if $F$ is doubling, then
\[
\ad  F = \inf \left\{ \urd \nu \,  \colon \, \nu \text{ is a Borel probability measure fully supported on } F\right\}
\]
and
\[
\ld F = \sup \left\{ \lrd \nu \,  \colon \, \nu \text{ is a Borel probability measure fully supported on } F \right\},
\]
see \cite{anti2} and the references therein.  Having finite Assouad dimension is equivalent to being doubling,  and having strictly positive lower dimension is equivalent to being uniformly perfect.  J\"arvi and Vuorinen  \cite[Theorem 5.5]{uniperfect} proved that limit sets of finitely generated Kleinian groups are uniformly perfect and so it is natural to pursue a  quantitative version of this result where one computes the lower dimension explicitly.   Indeed, uniform perfectness has been investigated extensively in the context of Kleinian limit sets, see \cite{sugawa} and the references therein.

We close this section with the observation that the regularity dimensions (and therefore the Assouad and lower dimensions) are preserved under M\"obius transformations.  Although simple, this observation is important since properties of the limit set and Patterson-Sullivan measure should be preserved under the action of $\isom$ on $\mathbb{S}^{d}$ and also should be independent of the chosen model of hyperbolic space.  The observation that the Assouad dimension is preserved under general M\"obius transformations can be found in \cite[Theorem A.10]{Luukkainen} and a similar argument establishes the analogous result for lower dimension. For the action of $\isom$ on $\mathbb{S}^d$ the situation is already very simple since each element $g \in \isom$ is bi-Lipschitz  on  $\mathbb{S}^d$ (the bi-Lipschitz constants are not uniform over $\isom$, but this does not matter) and  so the dimensions of sets and measures supported on $\mathbb{S}^{d}$ are clearly preserved by $\isom$.  For general M\"obius transformations $g:\overline{\mathbb{R}^{d}} \to \overline{\mathbb{R}^{d}}$, if $\mu$ is a Borel probability measure supported on $\mathbb{R}^{d}$, then, provided $\mu (\{g^{-1}(\infty)\}) = 0$, the pushforward measure $g(\mu) = \mu \circ g^{-1}$ is a Borel probability measure supported on $\mathbb{R}^{d}$ and one can show that the regularity dimensions of $\mu$ and $g(\mu)$ coincide.  We do not rely on this fact, but point it out to reassure readers that our results are independent of how we model  hyperbolic space.

\section{Results}

Throughout this section $\Gamma < \isom$ will be a non-elementary geometrically finite Kleinian group acting on $\mathbb{D}^{d+1}$.  Also, $\ls \subseteq  \mathbb{S}^{d}$ will denote the  limit set of $\Gamma$,  $\ps$ the associated Patterson-Sullivan measure, and $\delta(\Gamma)$ the Poincar\'e exponent. In order to remove simple cases, we start with the following immediate consequence of the global measure formula (\ref{global}).

\begin{thm}
If $\Gamma$ is a non-elementary geometrically finite Kleinian group which is parabolic free, then $\ps$ is $\delta(\Gamma)$-Ahlfors-David regular and therefore
\[
\ad \ls = \ld \ls = \urd \ps = \lrd \ps = \delta(\Gamma).
\]
\end{thm}

In light of this result, we assume from now on that $\Gamma$ contains a parabolic element and write $1 \leq \kmin \leq \kmax \leq d$ to denote the minimal and maximal ranks of parabolic fixed points, respectively. 

Our first result gives precise formulae for the regularity dimensions of the Patterson-Sullivan measure associated with a geometrically finite Kleinian group.

\begin{thm} \label{reg}
Let $\Gamma< \isom$ be a non-elementary geometrically finite Kleinian group which is not parabolic free.  The upper and lower regularity dimensions of the Patterson-Sullivan measure are continuous and piecewise linear in the Poincar\'e exponent both with a single phase transition at $(\kmin+\kmax)/2$.  In particular, 
\[
\urd \ps \ = \ \max\left\{ \kmax, \, 2 \delta(\Gamma) - \kmin \right\}
\]
and
\[
\lrd \ps \ = \ \min\left\{ \kmin, \, 2 \delta(\Gamma) - \kmax \right\}.
\]
\end{thm}

We prove Theorem \ref{reg} in Section \ref{regproof}.  We now turn our attention to the related question of the Assouad and lower dimensions and state our main result.

\begin{thm} \label{ass}
Let $\Gamma< \isom$ be a non-elementary geometrically finite Kleinian group  which is not parabolic free.  The Assouad and lower dimensions of $\ls$ are functions of the Poincar\'e exponent of convergence and the maximal and minimal ranks of parabolic fixed points.  In particular,
\[
\ad \ls =  \max\{ \kmax, \,   \delta(\Gamma)\}
\]
and
\[
\ld \ls =  \min\{ \kmin, \, \delta(\Gamma) \}.
\]
\end{thm}

We prove Theorem \ref{ass} in Section \ref{assproof}.  We emphasise here that, even though the Hausdorff, packing and upper and lower box dimensions of $\ls$ always coincide with the Poincar\'e exponent, the Assouad and lower dimensions may not.

\subsection{Applications and observations}

In this section we discuss several consequences of our results, hopefully demonstrating their relevance in other contexts.

\textbf{The Patterson-Sullivan measure is `uniformly perfect':}  It follows from Theorem \ref{reg} that the lower regularity dimension of the Patterson-Sullivan measure for a non-elementary geometrically finite Kleinian group is always strictly positive.  This can be viewed as a measure theoretic analogue of the result in \cite[Theorem 5.5]{uniperfect} that the support of such measures are uniformly perfect, that is, have strictly positive lower dimension.

\textbf{Consequences and characterisations of full Assouad dimenion:}  Having full Assouad dimension (i.e. equal to that of the ambient space) is a strong property with numerous consequences.  In particular, not having  full Assouad dimension is equivalent to being porous and if the Assouad dimension is full, then so is the conformal Assouad dimension, that is, the Assouad dimension may not be lowered by quasi-symmetric transformations, see \cite{fraseryu, mackaytyson}.  Theorem \ref{ass}, combined with the deep result of Tukia \cite[Theorem D]{tukia} that $\delta(\Gamma)< d$ if and only if the limit set is not the entire boundary, provides the following precise characterisation of when the Assouad dimension of $\ls$ has full Assouad dimension.

\begin{cor}
If  $\Gamma < \isom$ is geometrically finite and  $\ls \neq \mathbb{S}^d$, then the following are equivalent:
\begin{enumerate}
\item the limit set has full Assouad dimension, that is $\ad \ls = d$,
\item there exists a cusp of maximal rank, that is $\kmax = d$,
\item for any quasi-symmetric transformation $\phi$, we have $\ad \phi(\ls) = d$,
\item the limit set is non-porous.
\end{enumerate}
\end{cor}

In fact our arguments prove that the conformal Assouad dimension is always bounded below by $\kmax$ and is equal to $\kmax$ whenever $\delta(\Gamma) \leq \kmax$.

\textbf{Invariant measures with optimal dimensions:}  The interplay between dynamically invariant sets and measures is central to  the dimension theory of dynamical systems with a  natural question  being: when does a given invariant set support an (ergodic) invariant measure which realises its dimension?  This question can then take on different flavours depending on the dimensions involved.  Concerning Hausdorff dimension,  `realising the dimension'  usually means that the Hausdorff dimension of the measure equals the Hausdorff dimension of the set (a measure of maximal dimension).   For Assouad and lower dimension, `realising  the dimension'   means that the upper/lower regularity dimension of the measure equals the Assouad/lower dimension of the set (a measure of minimal/maximal dimension). 

 It is particularly interesting to us whether or not an invariant measure can simultaneously realise  all three of these dimensions when they are distinct. Previous examples seem to support a negative answer to this question.  For example, consider the self-affine carpets of Bedford-McMullen \cite{bedford, mcmullen}, which are invariant under the toral endomorphism  $(x,y) \mapsto (mx,ny)$.  It is well-known that there exists a unique invariant probability measure of maximal Hausdorff dimension, namely the \emph{McMullen measure}.  In the case where the carpet does not have uniform fibres, the Assouad, Hausdorff, and lower dimensions are necessarily distinct, see \cite[Corollary 2.14]{Fraser}.    It follows from \cite[Theorem 2.6]{fraserhowroyd} that the upper and lower regularity dimensions of the McMullen measure are always distinct from the Assouad and lower dimensions of the carpet, provided the \emph{very strong separation condition (VSSC)} is satisfied and the carpet does not have uniform fibres.  There are invariant  measures, introduced in \cite[Theorem 2.3]{fraserhowroyd2}, which simultaneously realise the Assouad and lower dimensions, provided the VSSC is satisfied.  These measures are known as \emph{coordinate uniform measures}, but are necessarily distinct from the McMullen measure and so do not realise the Hausdorff dimension.

We can provide the first example of a dynamically invariant measure which simultaneously realises the  (distinct) lower, Hausdorff and Assouad dimensions of its support, thus answering the above question in the affirmative.  However, as we shall see, this simultaneous realisation is still very rare in this context.  Note that $\ps$ \emph{always} realises the Hausdorff dimension of $\ls$.

\begin{cor} \label{achieve}
Let $\Gamma< \isom$ be a non-elementary geometrically finite Kleinian group which is not parabolic free.  Then
\[
\urd \ps = \ad \ls  \  \  \Leftrightarrow \ \   \delta(\Gamma) \leq  (\kmin+\kmax)/2
\]
and
\[
\lrd \ps = \ld \ls \  \  \Leftrightarrow \   \ \delta(\Gamma) \geq (\kmin+\kmax)/2.
\]
Therefore we have
\[
\lrd \ps = \ld \ls < \hd \ps = \hd \ls < \urd \ps = \ad \ls 
\]
if and only if $\kmin<\kmax$ and $\delta(\Gamma) = (\kmin+\kmax)/2$.
\end{cor}

Coming up with an explicit example where $\kmin<\kmax$ and $\delta(\Gamma) = (\kmin+\kmax)/2$ is not straightforward but can be achieved as a subgroup of $\text{Con}(2)$ by starting with a group with cusps of both possible ranks (1 and 2) and small $\delta(\Gamma)$, i.e. close to $\kmax/2 = 1$.  Then by varying some hyperbolic elements in a region which does not interfere with the cusps one can slowly increase $\delta(\Gamma)$ achieving $\delta(\Gamma) = 3/2$ at some point by the intermediate value theorem.  We do not pursue the details.

\textbf{Relationships between dimensions:}  It is a common consideration in dimension theory to identify possible relationships between dimensions in particular contexts, see for example \cite[Section 4]{Fraser}.  A succinct corollary of our main results demonstrates the following precise dichotomy for the dimensions of  limit sets of \emph{Fuchsian} groups, which is somewhat reminiscent of a known dichotomy for the dimensions of self-similar (and self-conformal) sets in $\mathbb{R}$, see \cite[Theorem 1.3]{fraserrobinson} and also \cite[Theorems A and B]{sascha} for the conformal case.

\begin{cor}
Let $\Gamma < \mathrm{Con}(1) \cong \textup{PSL}(2, \mathbb{R})$ be a non-elementary geometrically finite Fuchsian group with limit set $L(\Gamma)$ a proper subset of $ \mathbb{S}^1$. 
\begin{enumerate}
\item If $\Gamma$ is parabolic free, then
\[
0< \ld \ls = \hd \ls  = \ad \ls = \delta(\Gamma) < 1 .
\]
\item If $\Gamma$  contains a parabolic element, then
\[
1/2 < \ld \ls = \hd \ls = \delta(\Gamma) < \ad \ls =  1 .
\]
\end{enumerate}
\end{cor}

We provide one more example,  which could be contrasted with, for example, a dichotomy observed by Mackay \cite[Theorem 1.1]{mackay} and \cite[Corollary 2.14]{Fraser}, which shows that for Bedford-McMullen carpets one either has the Assouad, lower and Hausdorff dimensions all equal or all distinct.

\begin{cor}
Let $\Gamma < \isom$ be a  non-elementary geometrically finite Kleinian group with at least one cusp, but with uniform cusp ranks, that is $\kmin = \kmax\geq 1$.  Then either
\[
\ld \ls = \hd \ls < \ad \ls \qquad \text{or} \qquad \ld \ls < \hd \ls = \ad \ls.
\]
\end{cor}

\textbf{Relationships with local dimensions:} The regularity dimensions are related to the local dimensions.  The upper local dimension of a Borel measure $\mu$ at $x \in \text{supp}(\mu)$ is defined by
\[
\overline{\dim}_{\text{loc}}(\mu,x)=\limsup_{r\rightarrow 0} \frac{\log \mu(B(x,r))}{\log r}.
\]
The lower local dimension $\underline{\dim}_{\text{loc}}(\mu,x)$ is defined in a similar way, replacing $\limsup$ with $\liminf$.  It is straightforward to see that for any measure $\mu$
\[
\lrd \mu \leq \inf_{x} \lld (\mu, x) \leq  \sup_{x} \uld (\mu, x)  \leq \urd \mu,
\]
see for example \cite[Theorem 2.1]{fraserhowroyd} for the upper regularity dimension case.  Moreover, equality between any of the terms above can be interpreted as some form of homogeneity of $\nu$.  Such homogeneity is rare for Patterson-Sullivan measures associated to Kleinian groups with parabolic elements.

\begin{prop} \label{localdims}
Let $\Gamma< \isom$ be a non-elementary geometrically finite Kleinian group, which is not parabolic free.  Then
\[
\sup_{z \in \ls} \uld (\ps, z) =  \max\left\{\delta(\Gamma), \,  2\delta(\Gamma) - \kmin\right\}
\]
and
\[
\inf_{z \in \ls} \lld (\ps, z) = \min\left\{\delta(\Gamma), \,  2\delta(\Gamma) - \kmax\right\}.
\]
In particular,
\[
\urd \ps = \sup_{z \in \ls} \uld (\ps, z) \  \  \Leftrightarrow \ \   \delta(\Gamma) \geq (\kmin+\kmax)/2
\]
and
\[
\lrd \ps = \inf_{z \in \ls} \lld (\ps, z) \  \  \Leftrightarrow \   \ \delta(\Gamma) \leq (\kmin+\kmax)/2.
\]
\end{prop}

The calculation of the extremal upper and lower dimensions is not new, see for example \cite{stratmann}, but we include the explicit calculation for completeness.  However, we delay this until Section \ref{localdimsproof} since it relies on observations we make during the proof sections.

\subsection{Examples}

In order to provide a visual picture for the statements of Theorems \ref{reg} and \ref{ass}, we plot the dimensions in three  distinct  cases: $\kmin< \kmax/2$, $\kmax/2 < \kmin < \kmax$ and $\kmin=\kmax$.  These plots will be useful to keep in mind when reading the subsequent proofs.

\begin{figure}[H]
  \centering
  \includegraphics[width= \textwidth]{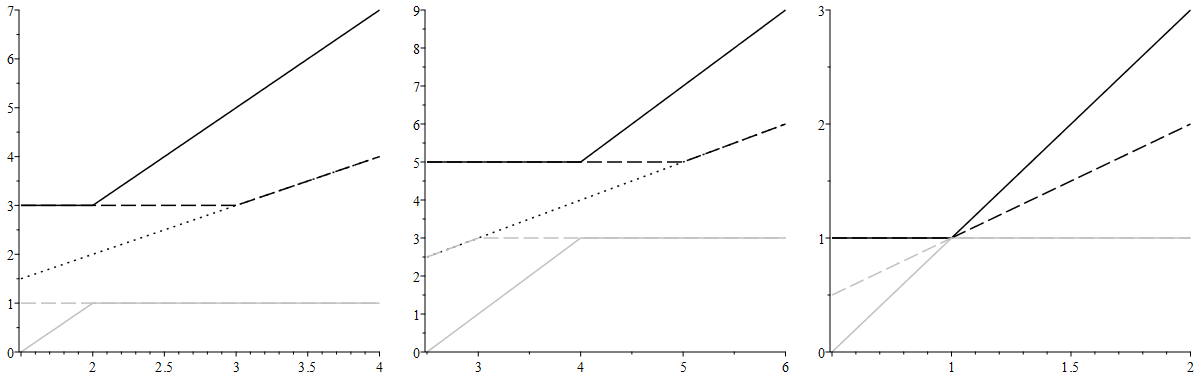}
\caption{Three plots showing the dimensions we study as functions of $\delta(\Gamma) \in (\kmax/2, d]$.  The regularity dimensions of $\ps$ are plotted with solid lines, the Assouad and lower dimensions of $\ls$ are plotted with dashed lines, and $\delta(\Gamma)$ is plotted with a dotted line (for reference).  The upper regularity and Assouad dimensions are plotted in black and the lower regularity and lower dimensions are plotted in grey.  Left: $\kmin = 1< \kmax = 3$ and $d=4$.  Centre:  $\kmin = 3< \kmax = 5$ and $d=6$.  Right: $\kmin =  \kmax = 1$ and $d=2$. }
\label{fig:plots}
\end{figure}

Also for aesthetic purposes, we close this section by discussing a famous example.  The \emph{Apollonian gasket} or \emph{Apollonian circle packing}, see Figure \ref{fig:app}, is a well-known geometric object formed by starting with 4 mutually tangent circles lying in $\mathbb{C}$, one containing the other three, and then inductively adding in circles of the largest possible radius which lie tangent to three previously added circles.   See \cite{pollicott} for an interesting discussion of Apollonian packings ranging from their genesis to problems at the forefront of modern mathematics and \cite{parker2} for more on the visualisation of Apollonian circle packings as well as  other Kleinian limit sets. It is well-known that given any two circle packings formed in this way there is a M\"obius transformation $g \in \text{PSL}(2,\mathbb{C} )$  taking one to the other, that is, there is a unique circle packing up to M\"obius images.  Therefore we may talk about \emph{the} Apollonian circle packing and note that it is the limit set of a geometrically finite Kleinian group $\Gamma < \text{PSL}(2,\mathbb{C} ) \cong \text{Con}(2)$, sometimes known as the \emph{Apollonian group}.  

The parabolic fixed points are the points of mutual tangency between two circles in the packing and it is straightforward to see that the rank of each of these points is 1, and therefore $\kmin = \kmax = 1$. Estimating the Poincar\'e exponent for this group is difficult, but has received a lot of attention in the literature and very good bounds are now available.  In particular, $\delta(\Gamma) \approx 1.305 \dots$, see  \cite{mcmullen2}.  We note that the Poincar\'e exponent can also be computed as the circle packing exponent, which is somewhat easier to handle computationally, see \cite{boyd, parker1}. Therefore
\[
\urd \ps = 2 \delta(\Gamma) - 1 \approx 1.61 \dots 
\]
\[
\lrd \ps = \ld \ls = 1
\]
\[
\ad \ls = \delta(\Gamma) \approx 1.305 \dots
\]

\begin{figure}[H]
  \centering
  \includegraphics[width= 0.5\textwidth]{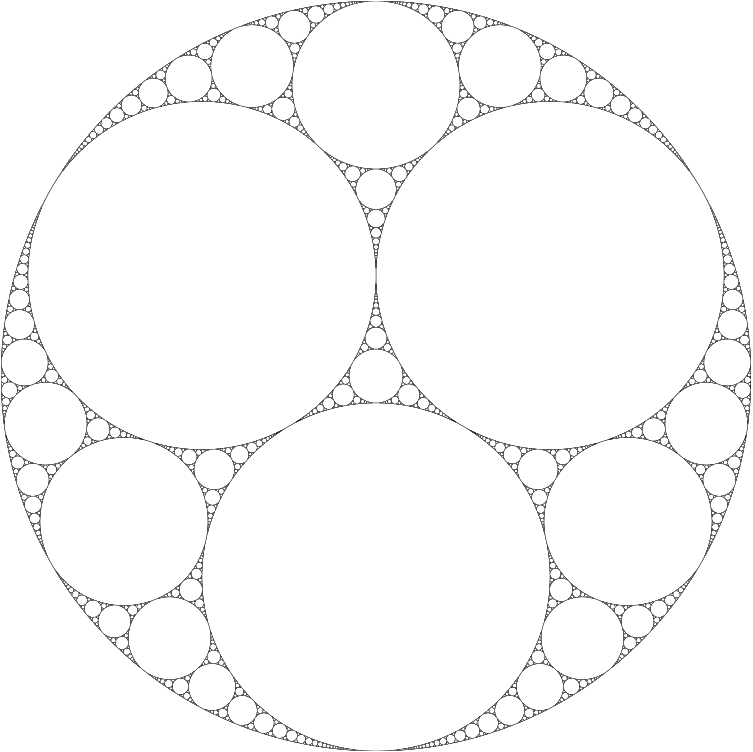}
\caption{An Apollonian circle packing viewed as the limit set of the Apollonian group acting on $\mathbb{H}^3$.  }
\label{fig:app}
\end{figure}

\subsection{The geometrically infinite case} \label{infinite}

In this section we briefly discuss the geometrically infinite case.  Limit sets of non-elementary geometrically infinite Kleinian groups are not as well-understood as the geometrically finite case. They can also exhibit many different features, not present in the geometrically finite case, for example, one generally has $ \delta(\Gamma) \leq \hd \ls \leq \bd \ls$, but either or both of these inequalities can be strict.  It was recently shown by Falk and Matsuzaki \cite[Theorem 1]{kurt} that the (upper) box dimension of the limit set is given by the \emph{convex core entropy}, $h_c(\Gamma)$, see \cite[Definition 3.1]{kurt}.  This result, combined with the observation that our proof of the \emph{lower} bound for the Assouad dimension of the limit set does not use geometric finiteness, yields the following estimate.
\begin{cor}
If $\Gamma < \isom$ is a non-elementary Kleinian group, then
\[
\ad \ls \geq   \max\{ \kmax, h_c(\Gamma)\}.
\]
\end{cor}
It is natural to ask if equality holds here, but this turns out not to be true in general. We demonstrate this by example at the end of this section.

A weakening of geometric finiteness is the concept of \emph{conformal finiteness}, introduced by  Chang, Qing and  Yang \cite[Definition 3.2]{chang}, which extends the older notion of \emph{analytic finiteness} for  subgroups of $\text{Con}(2)$.  In particular, all finitely generated Kleinian groups $\Gamma < \text{Con}(2)$ are analytically and conformally finite (this is known as Ahlfors finiteness theorem and is known to fail in higher dimensions, see \cite{kapovich}).   It is shown in \cite[Theorem 0.1]{chang} that  if $\Gamma < \isom$ is conformally finite and $\hd \ls < d$, then it is geometrically finite.  This result was proved for $d \leq 2$ by Bishop and Jones \cite[Theorem 1.2]{bishopjones}.  In particular, the result of Chang-Qing-Yang   combined with our Theorem \ref{ass} provides the following corollary.
\begin{cor}
If $\Gamma < \isom$ is a non-elementary conformally finite Kleinian group, then
\[
\ad \ls = \max\{ \kmax, \, \hd \ls\}.
\]
\end{cor}

Finally we present a simple example illustrating the wildness of infinitely generated Kleinian groups, see \cite{matsu} for discussion of the Hausdorff dimension.  Specifically, for any $0<\alpha<\beta<1$, we sketch the construction of an infinitely generated Fuchsian group $\Gamma < \text{PSL}(2, \mathbb{R}) \cong \text{Con}(1)$ with
\[
 \ld \ls = 0  < \hd \ls \leq \alpha < \beta \leq \bd \ls < \ad \ls = 1.
\]
Moreover, $\Gamma$ will be parabolic free and so this shows that the Assouad dimension can be large for reasons other than parabolic points in the infinitely generated case. Of course there is a natural duality between parabolic systems  and infinitely generated systems (for example, via `inducing schemes') and so it is really just two sides of the same coin.  By the result of Falk and Matsuzaki mentioned above this example also demonstrates that $\ad \ls > \max\{\kmax, h_c(\Gamma)\}$ is possible  in the infinitely generated case.  It also shows that limit sets of infinitely generated Fuchsian groups need not be uniformly perfect (they can have lower dimension equal to 0).  This observation is not new, see for example \cite[Example 7.1]{sugawa}.  Finally, this  example also demonstrates that for  infinitely generated $\Gamma < \isom$, the difference $\ad \ls - \hd \ls$ can approach $d$, whereas in the geometrically finite case it can only approach $d/2$, see Theorem \ref{ass} noting that $\delta(\Gamma) > \kmax/2$.  In the geometrically finite case the (potentially) larger difference $\ad \ls - \ld \ls$ is bounded above by $d-1$ and this bound is achieved precisely when $\kmin=1<d=\kmax$, whereas in the geometrically infinite case it can be $d$.

Fix $0<\alpha<\beta<1$ and set $\gamma = 1/\beta - 1>0$.  For integer $n \geq 1$, let $x_n = x_n(\gamma) = 1/n^\gamma$ and $0< r_n = r_n(\alpha,\gamma) < e^{-n}$ be very small radii, chosen so that the balls $B(x_n,r_n)$ are pairwise disjoint.  Let $g_n: \mathbb{H}^2 \to \mathbb{H}^2$  be defined by reflecting in the boundary of the ball $B(x_n, r_n)$ (an orientation reversing M\"obius transformation).  Since the balls $B(x_n,r_n)$ are pairwise disjoint the group
\[
\Gamma' = \langle g_n \ :  \ n \geq 1 \rangle
\]
is discrete.  Moreover, $\Gamma'$ has an index 2 subgroup $\Gamma < \Gamma'$ consisting of orientation preserving isometries which is a Fuchsian  subgroup of $\text{PSL}(2, \mathbb{R}) $.  It is easy to see that
\[
\ls \subseteq \cup_n B(x_n, r_n) \cup \{0\}
\]
and that for all $n \geq 1$, $\ls \cap B(x_n, r_n) \neq \emptyset$.  By considering the dimensions of the set of centres  $\{x_n \}_{ n \geq 1}$ this is already enough to prove that
\[
\bd \ls \geq \frac{1}{1+\gamma} = \beta
\]
and $\ad \ls = 1$.  Moreover, since the radii $r_n$ decay exponentially and the gaps between the balls only decay polynomially, it follows that $\ld \ls = 0$.  Indeed, $N_{e^{-n}}(B(x_n,n^{-(\gamma+1)})) \lesssim 1$.  Finally, $\hd \ls$ can be made arbitrarily small by choosing the radii $r_n$ small enough. With a little more work one can show that the box dimension is indeed controlled by the box dimension of the set of fixed points of the generators, and therefore is precisely $\beta = h_c(\Gamma) < 1 = \ad \ls$.  See \cite{urbanski} for more general settings where the box dimension of infinitely generated limit sets  is controlled by the maximum of the Hausdorff dimension and the box dimension of the set of fixed points.

\section{Notation and preliminary results}

Throughout the rest of the paper we will write $A \lesssim B$ to mean that there exists a universal constant $c\geq 1$ such that $A \leq cB$.  In particular, $c$ is allowed to depend on parameters fixed in the hypotheses of the theorems given above, such as the group $\Gamma$, and  ambient spatial dimension $d$.  The constant $c$ is \emph{not} allowed to depend on parameters introduced during the proof, most importantly the scales $R,r$ or (logarithmic) scales $T, t$ or on particular points $z \in \mathbb{D}^{d+1} \cup \mathbb{S}^d$.  We write $A \gtrsim B$ to mean $B \lesssim A$ and $A \approx B$ to mean $A \lesssim B$ and $A \gtrsim B$.

We begin by reformulating the statement of the  global measure formula, which also serves as a crucial example using the notation described above.   It follows immediately from (\ref{global}) that 
\begin{equation} \label{global3}
\ps(B(z,e^{-t}))  \ \approx \ \exp(-t \delta(\Gamma)- \rho(z,t) (\delta(\Gamma) - k(z,t)))
\end{equation} 
 for all $z \in \ls$ and $t>0$  where $B(z,e^{-t})$ is the Euclidean ball centred at $z$ with radius $e^{-t}$.   Note that the implied constants  only depend on the group $\Gamma$ and the choice of standard horoballs (which we may assume depends on the group). The implied constants do not depend on $z$ or $t$.

Since the global measure formula is most conveniently expressed in terms of logarithmic scales $t>0$, that is, balls with radius $e^{-t}$, we adopt this convention whenever we use (\ref{global3}).  In particular, when computing the upper and lower regularity dimensions we will use a `large' scale $R = e^{-t}$  and a `small' scale $r=e^{-T}$ for $T \geq t >0$.  Therefore to prove that $\urd \ps \leq \alpha$, say, it suffices to prove that
\[
\frac{\ps (B(z,e^{-t}))}{\ps (B(z,e^{-T}))} \ \lesssim \ \left(\frac{e^{-t}}{e^{-T}} \right)^{\alpha}
\]
for all $z \in \ls$ and $T > t >0$, whereas to prove that $\urd \ps \geq \beta$, say, it suffices to prove that
\[
\frac{\ps (B(z,e^{-t}))}{\ps (B(z,e^{-T}))} \ \gtrsim \ \left(\frac{e^{-t}}{e^{-T}} \right)^{\beta}
\]
for infinitely many $z \in \ls$ and $T > t >0$ with $T-t \to \infty$.

We may assume without loss of generality that $ \mathbf{0} \notin H_p$ for all $p$, which means that $\rho(z,t) \leq \hm (z_t,0) \leq t$ for all $z \in \ls$ and $t >  0$.  The rest of this section is devoted to establishing simple estimates for the `escape functions' $\rho(z,t)$, which will be used throughout the rest of the paper.

\begin{lma}[Quick escape lemma]\label{escape}
Let $t_1, t_2 > 0$ and $z \in \ls$.  If $z_{t_1}$ and $z_{t_2}$ do not lie in a common standard horoball, then
\[
\rho(z, t_1)+\rho(z, t_2)  \ \leq \  |t_1-t_2 |
\]
and if $z_{t_1}$ and $z_{t_2}$ do lie in a common standard horoball, then
\[
|\rho(z,t_1)-\rho(z,t_2) |  \ \leq \  |t_1-t_2 |.
\]
\end{lma}

\begin{proof}
If $z_{t_1}$ and $z_{t_2}$ do not lie in the same horoball, then we can find $t_0$ lying between  $t_1$ and $t_2$ such that $z_{t_0}$ does not lie in the interior of any horoball. It follows that
\begin{eqnarray*}
\rho(z, t_1) +\rho(z, t_2) \  \leq  \   \hm(z_{t_1}, z_{t_0}) +  \hm(z_{t_0}, z_{t_2})   &=&  \hm(z_{t_1}, z_{t_2})   \\ &=&   |\hm(z_{t_1},0)  - \hm(0, z_{t_2}) | \\ &=&   |t_1-t_2 |
\end{eqnarray*}
since $ \mathbf{0}, z_{t_1}, z_{t_2},  z_{t_0}$ lie on the same geodesic.

Now suppose $z_{t_1}, z_{t_2} \in H_p$ for some standard horoball $H_p$.  Then $z_{t_1}$ can escape $H_p$ by first going to $z_{t_2}$ and then escaping from there via the most efficient route.  Therefore
\[
\rho(z,t_1) \leq \hm(z_{t_1}, z_{t_2}) + \rho(z,t_2) =   |t_1-t_2 | + \rho(z,t_2)
\]
and the result follows by symmetry.
\end{proof}

\begin{lma}[Parabolic centre lemma]\label{parapoint}
Suppose $p$ is a parabolic fixed point associated to a standard horoball $H_p$.  Then $\rho(p,t) \sim t$ as $t \to \infty$ and for all sufficiently large $t>0$ we have $k(p,t) = k(p)$.
\end{lma}

\begin{proof}
Let $s>0$ be such that $p_s$ is the point of intersection of the ray joining $ \mathbf{0}$ and $p$ with the boundary of the horoball $H_p$. It follows that $p_t \in H_p \Leftrightarrow t \geq s$ and therefore for $t \geq s$, we have $k(p,t) = k(p)$.  Moreover, since the geodesic joining $p_t$ and $p_s$ is normal to the boundary of $H_p$,
\[
1 \geq \frac{\rho(p,t)}{t} =\frac{\hm(p_t, p_s)}{t} =  \frac{t-s}{t} \to 1
\]
as $t \to \infty$, as required.
\end{proof}

We will also need a version of Lemma \ref{escape} for when the point $z$ is not fixed. We state and prove this version separately for clarity.

\begin{lma}[Quick escape lemma II]\label{escape2}
Let $T>t>0$ and $x,y  \in \ls$ with $\| x-y \| \leq 2 e^{-t}$.  If $x_t$ and $y_T$ do not lie in a common standard horoball, then
\[
\rho(x,t) + \rho(y,T)  \ \leq  \ T-t+10
\]
and if $x_t$ and $y_T$ do lie in a common standard horoball, then
\[
|\rho(x,t)-\rho(y,T) |  \ \leq   \  T-t + 10.
\]
\end{lma}

\begin{proof}
Suppose that $x_t$ and $y_T$ do not lie in a common standard horoball and assume without loss of generality that at least one of $x_t$, $y_T$ lies in the interior of some  horoball.   Therefore there must be a point on the geodesic joining $x_t$ and $y_T$ which lies on the boundary of this horoball.  It follows that
\[
\rho(x,t)+\rho(y,T)  \  \leq \    \hm(x_t, y_T) \  \leq \   \hm(y_t,y_T)  +\hm(x_t,y_t)  \ \leq \   (T-t)  +  10  
\]
as required. The  estimate $\hm(x_t,y_t)  \leq 10$ is not sharp and follows by basic hyperbolic geometry estimates using $\| x-y \| \leq 2e^{-t}$.  See, for example,  \cite[Proposition 4.3]{anderson}.   

Now suppose $x_t$ and $y_T$ do lie in a common standard horoball.  Then, combining ideas from the proof of Lemma \ref{escape} and the above argument, we get
\[
|\rho(x,t)-\rho(y,T) | \ \leq  \  \hm(x_t, y_T)  \  \leq \ \hm(y_t,y_T)  +\hm(x_t,y_t)  \ \leq \   (T-t)  +  10  
\]
completing the proof.
\end{proof}

\section{The regularity dimensions of $\ps$: proof of Theorem \ref{reg}} \label{regproof}

\subsection{The upper regularity dimension}

\subsubsection{Upper bound: $\urd \ps \leq \max\{\kmax, \,  2 \delta(\Gamma) - \kmin\}$}

 Let $z \in \ls$ and $T \geq t > 0$.  It follows from (\ref{global3}) that
\begin{eqnarray*}
\frac{\ps (B(z,e^{-t}))}{\ps (B(z,e^{-T}))} & \lesssim &  \frac{\exp(-t \delta(\Gamma)- \rho(z,t) (\delta(\Gamma) - k(z,t)))}{\exp(-T \delta(\Gamma)- \rho(z,T) (\delta(\Gamma) - k(z,T)))} \\ \\
& = &  \left(\frac{e^{-t}}{e^{-T}} \right)^{\delta(\Gamma)}\frac{ \exp( \rho(z,t) (k(z,t)-\delta(\Gamma) ))}{\exp(\rho(z,T) (  k(z,T)-\delta(\Gamma)))}  \qquad (\dagger)
\end{eqnarray*}
If $z_t$ and $z_T$ lie in the same standard horoball $H_p$, then, continuing from $(\dagger)$, we get
\begin{eqnarray*}
\frac{\ps (B(z,e^{-t}))}{\ps (B(z,e^{-T}))} &\lesssim &   \left(\frac{e^{-t}}{e^{-T}} \right)^{\delta(\Gamma)}\frac{ \exp( \rho(z,t) (k(p)-\delta(\Gamma) ))}{\exp(\rho(z,T) (  k(p)-\delta(\Gamma)))}  \\ \\
& = &  \left(\frac{e^{-t}}{e^{-T}} \right)^{\delta(\Gamma)} \exp \Big( (\rho(z,t)-\rho(z,T))(  k(p) - \delta(\Gamma)) \Big) \\ \\
& \leq &  \left(\frac{e^{-t}}{e^{-T}} \right)^{\delta(\Gamma)} \exp \Big( (T-t)  | k(p) - \delta(\Gamma)| \Big) \qquad \text{by Lemma \ref{escape}}  \\ \\
& =&  \left(\frac{e^{-t}}{e^{-T}} \right)^{\max\{ k(p), \, 2 \delta(\Gamma)- k(p)\}} \\ \\
& \leq& \left(\frac{e^{-t}}{e^{-T}} \right)^{\max\{ \kmax , \, 2 \delta(\Gamma)- \kmin\}}.
\end{eqnarray*}
Note that $\rho(z,T) \neq 0 \Rightarrow k(z,T) \geq \kmin$.  Therefore, if $z_t$ and $z_T$ do not lie in the same standard horoball, then, returning to $(\dagger)$, we get
\begin{eqnarray*}
\frac{\ps (B(z,e^{-t}))}{\ps (B(z,e^{-T}))}& \lesssim &  \left(\frac{e^{-t}}{e^{-T}} \right)^{\delta(\Gamma)}\frac{ \exp( \rho(z,t) (\kmax-\delta(\Gamma) ))}{\exp(\rho(z,T) (  \kmin-\delta(\Gamma)))}  \\ \\
& \leq &  \left(\frac{e^{-t}}{e^{-T}} \right)^{\delta(\Gamma)} \exp \Big( (\rho(z,t)+\rho(z,T)) \max\{\kmax-\delta(\Gamma) ,   \delta(\Gamma) - \kmin\} \Big) \\ \\
& \leq &  \left(\frac{e^{-t}}{e^{-T}} \right)^{\delta(\Gamma)} \exp \Big( (T-t) \max\{\kmax-\delta(\Gamma) , \,   \delta(\Gamma) - \kmin\} \Big) \\ \\
&\,& \hspace{80mm} \qquad \text{by Lemma \ref{escape}} \\ \\
&= & \left(\frac{e^{-t}}{e^{-T}} \right)^{ \max\{\kmax, \,  2 \delta(\Gamma) - \kmin\} } 
\end{eqnarray*}
It follows that $\urd \ps \leq \max\{\kmax, \,  2 \delta(\Gamma) - \kmin\}$, as required.

\subsubsection{Lower bound: $\urd \ps \geq  \max\{\kmax, \,  2 \delta(\Gamma) - \kmin\}$}

We first show that $\urd \ps \geq   2 \delta(\Gamma) - \kmin$ provided $\delta(\Gamma) \geq \kmin$ by considering parabolic fixed points.  Suppose  $\delta(\Gamma) \geq \kmin$,  choose $p \in \ls$ to be a parabolic fixed point of minimal rank $k(p) = \kmin$, and let $\varepsilon \in (0,1)$.  By Lemma \ref{parapoint} it follows that for $T>0$ sufficiently large we have
\begin{eqnarray*}
\frac{\ps (B(p,1))}{\ps (B(p,e^{-T}))} & \gtrsim &  \frac{1}{\exp(-\delta(\Gamma) T+\rho(p,T) ( \kmin-\delta(\Gamma)))} \\ \\
  & \geq &   \frac{1}{\exp(-\delta(\Gamma) T+(1-\varepsilon)T( \kmin-\delta(\Gamma)))} \\ \\
&=&  \left(\frac{1}{e^{-T}} \right)^{\delta(\Gamma) - (1-\varepsilon)(\kmin - \delta(\Gamma))} 
\end{eqnarray*}
which proves that $\urd \ps \geq 2\delta(\Gamma)-\kmin - \varepsilon( \delta(\Gamma)-\kmin )$ and letting $\varepsilon \to 0$ provides the desired lower bound.

Showing that $\urd \ps \geq  \kmax$ is more subtle since we cannot keep the point $z$ fixed.  This reflects the fact that this bound does not come from the local dimensions, see Proposition \ref{localdims}.  Suppose $\delta(\Gamma) \leq \kmax$, choose $p$ to be a parabolic fixed point of maximal rank $k(p) = \kmax$, and let $n \in \mathbb{Z}^+$ be a very large integer.  Let $p \neq z_0 \in \ls$,  $f$ be a parabolic element fixing $p$, and choose $z = f^n(z_0) \in \ls$.  Observe that $z\neq p$ and $z\to p$ as $n \to \infty$.  We assume $n$ is large enough to guarantee that the geodesic ray from $ \mathbf{0}$ to $z$ passes through $H_p$.  Choose $T>0$ to be the larger of the two values for which $z_T$ lies on the boundary of $H_p$ (i.e. $z_T$ is the `exit point' from $H_p$).  For simplicity, we now restrict our attention to the 2-dimensional hyperplane $H(p,z,z_T)$ containing the points $p, z, z_T$ and restricted to $\mathbb{D}^{d+1}$.   Let $u$ be the point on the boundary of $H_p \cap H(p,z,z_T)$ (which is a circle) which is at hyperbolic distance 1 from $z_T$ and lies further away from $p$ than $z_T$ (in Euclidean terms). Choose $t \in (0,T)$ such that $z_t \in H_p$ and such that the geodesic joining $z_t$ and $u$ is normal to the boundary of $H_p$.  This uniquely defines $t$, see Figure \ref{pic11}.
\begin{figure}[H]
  \centering
  \includegraphics[width= 0.65\textwidth]{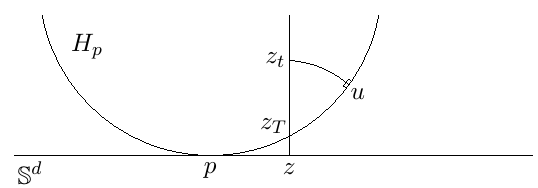}
\caption{Choosing $T$ and $t$. }
\label{pic11}
\end{figure}
\noindent It follows that
\[
\rho(z,t) = \hm (z_t,u) \geq \hm(z_t, z_T) - \hm(z_T, u) = T-t -1.
\]
 It follows from (\ref{global3}) and the fact that $\rho(z,T)=k(z,T)=0$ and $k(z,t) = \kmax$ by construction, that
\begin{eqnarray*}
\frac{\ps (B(z,e^{-t}))}{\ps (B(z,e^{-T}))} & \gtrsim &  \left(\frac{e^{-t}}{e^{-T}} \right)^{\delta(\Gamma)}\exp( \rho(z,t)  (\kmax - \delta(\Gamma) )) \\ \\
& \geq &  \left(\frac{e^{-t}}{e^{-T}} \right)^{\delta(\Gamma)}\exp(  (T-t -1)  (\kmax-\delta(\Gamma) )) \\ \\
&\gtrsim&  \left(\frac{e^{-t}}{e^{-T}} \right)^{\kmax}.
\end{eqnarray*}
Finally, basic hyperbolic geometry shows  that $T-t \to \infty$ as $n \to \infty$ and therefore $\urd \ps \geq \kmax$ as required. To see why $T-t \to \infty$, switch to the upper half-space model $\mathbb{H}^{2} \subseteq \mathbb{C}$ and assume that $p = \infty$.  Then the boundary of $H_p  \cap H(p,z,z_T)$ is simply a horizontal Euclidean line and the geodesic ray from $ \mathbf{0}$ (which is represented by $i \in \mathbb{H}^{2}$) to $z$ is an  arc of a circle which meets the boundary at right angles and (linearly) increases in radius  with $n$.  Observe that
\[
\log \frac{\text{Im}(z_t)}{\text{Im}(u)} = \hm (z_t,u) = \rho(z,t) \leq T-t
\]
and that $\text{Im}(u)$ is fixed but $\text{Im}(z_t)$ grows without bound in $n$, see Figure \ref{pic22}. 
\begin{figure}[H]
  \centering
  \includegraphics[width= 0.60\textwidth]{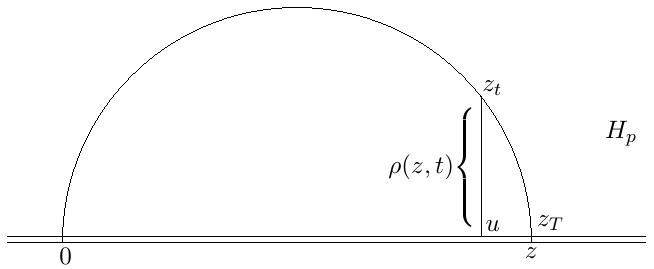}
\caption{An explanation of why $T-t \to \infty$.  For large $n$ the boundary of $H_p$ appears very close to the boundary of $\mathbb{H}^{2}$. }
\label{pic22}
\end{figure}

\subsection{The lower regularity dimension}

The calculation of  the lower regularity dimension is similar to the upper regularity dimension and so we only sketch the proof, leaving the details to the reader.  The lower bound closely follows the upper bound in the upper regularity case and the upper bound closely follows the lower bound in the upper regularity case.  The global measure formula (\ref{global3}) is again the key tool and the roles of $\kmin$ and $\kmax$ are reversed.

\subsubsection{Lower bound: $\lrd \ps \geq \min\{\kmin, \,  2 \delta(\Gamma) - \kmax\}$}

 Let $z \in \ls$ and $T \geq t > 0$.  It follows from (\ref{global3}) that if $z_t$ and $z_T$ lie in the same standard horoball $H_p$, then
\begin{eqnarray*}
\frac{\ps (B(z,e^{-t}))}{\ps (B(z,e^{-T}))} &\gtrsim  &   \left(\frac{e^{-t}}{e^{-T}} \right)^{\delta(\Gamma)} \exp \Big( (\rho(z,t)-\rho(z,T))(  k(p) - \delta(\Gamma)) \Big) \\ \\
& \geq &  \left(\frac{e^{-t}}{e^{-T}} \right)^{\delta(\Gamma)} \exp \Big( -(T-t)  | k(p) - \delta(\Gamma)| \Big) \qquad \text{by Lemma \ref{escape}}  \\ \\
& \geq&  \left(\frac{e^{-t}}{e^{-T}} \right)^{\min\{ \kmin , \, 2 \delta(\Gamma)- \kmax\}}.
\end{eqnarray*}
If $z_t$ and $z_T$ do not lie in the same standard horoball, then
\begin{eqnarray*}
\frac{\ps (B(z,e^{-t}))}{\ps (B(z,e^{-T}))}& \gtrsim  & \left(\frac{e^{-t}}{e^{-T}} \right)^{\delta(\Gamma)} \exp \Big( (\rho(z,t)+\rho(z,T)) \min\{\kmin-\delta(\Gamma) ,  \,  \delta(\Gamma) - \kmax\} \Big) \\ \\
& \geq &  \left(\frac{e^{-t}}{e^{-T}} \right)^{\delta(\Gamma)} \exp \Big((T-t) \min\{\kmin-\delta(\Gamma) ,   \delta(\Gamma) - \kmax\} \Big) \\ \\
 &\,& \hspace{10mm} \text{by Lemma \ref{escape} and since $\min\{\kmin-\delta(\Gamma) ,  \,  \delta(\Gamma) - \kmax\}  \leq 0$ } \\ \\
&\geq & \left(\frac{e^{-t}}{e^{-T}} \right)^{ \min\{\kmin, \,  2 \delta(\Gamma) - \kmax\} } .
\end{eqnarray*}
It follows that $\lrd \ps \geq \min\{\kmin, \,  2 \delta(\Gamma) - \kmax\}$, as required.

\subsubsection{Upper bound: $\lrd \ps \leq  \min\{\kmin, \,  2 \delta(\Gamma) - \kmax\}$}

 Suppose  $\delta(\Gamma) \leq \kmax$,  choose $p$ to be a parabolic fixed point of maximal  rank $k(p) = \kmax$, and let $\varepsilon \in (0,1)$.  By Lemma \ref{parapoint} it follows that for $T>0$ sufficiently large we have
\begin{eqnarray*}
\frac{\ps (B(p,1))}{\ps (B(p,e^{-T}))} & \lesssim &  \frac{1}{\exp(-\delta(\Gamma) T+\rho(p,T) ( \kmax-\delta(\Gamma)))} \\ \\
  & \leq &   \frac{1}{\exp(-\delta(\Gamma) T+(1-\varepsilon)T( \kmax-\delta(\Gamma)))} \\ \\
&=&  \left(\frac{1}{e^{-T}} \right)^{\delta(\Gamma) - (1-\varepsilon)(\kmax - \delta(\Gamma))} 
\end{eqnarray*}
which proves that $\lrd \ps \leq  2\delta(\Gamma)-\kmax + \varepsilon( \kmax -\delta(\Gamma) )$ and letting $\varepsilon \to 0$ provides the desired upper bound.

Analogous to the upper regularity dimension, showing that $\lrd \ps \leq  \kmin$ is more subtle since we cannot keep the point $z$ fixed.  Suppose $\delta(\Gamma) \geq \kmin$ and choose $p$ to be a parabolic fixed point of minimal rank $k(p) = \kmin$ and let $n \in \mathbb{Z}^+$ be a very large integer.  Let $p \neq z_0 \in \ls$,  $f$ be a parabolic element fixing $p$, and choose $z = f^n(z_0) \in \ls$. As above, choose $T>0$ such that  $z_T$ is the `exit point' from $H_p$ and choose $t \in (0,T)$ such that $z_t \in H_p$ and $\rho(z,t) \geq  T-t -1$.   It follows from (\ref{global3}) and the fact that $\rho(z,T)=k(z,T)=0$ and $k(z,t) = \kmin$ by construction, that
\begin{eqnarray*}
\frac{\ps (B(z,e^{-t}))}{\ps (B(z,e^{-T}))} & \lesssim &  \left(\frac{e^{-t}}{e^{-T}} \right)^{\delta(\Gamma)}\exp( \rho(z,t)  (\kmin- \delta(\Gamma) )) \\ \\
& \leq &  \left(\frac{e^{-t}}{e^{-T}} \right)^{\delta(\Gamma)}\exp(  (T-t -1)  (\kmin-\delta(\Gamma) )) \\ \\
&\lesssim&  \left(\frac{e^{-t}}{e^{-T}} \right)^{\kmin}.
\end{eqnarray*}
Observe, as before, that $T-t \to \infty$ as $n \to \infty$ and therefore $\lrd \ps \leq  \kmin$ as required.

\section{The dimensions of $\ls$: proof of Theorem \ref{ass}} \label{assproof}

\subsection{The Assouad dimension}

\subsubsection{Lower bound: $\ad \ls \geq   \max\{\kmax, \, \delta(\Gamma)\}$}

 Since $\ad \ls \geq \hd \ls = \delta(\Gamma)$ it suffices to prove that $\ad \ls \geq \kmax$.  Let $p \in \ls$ be a parabolic fixed point of maximal rank and  choose parabolic elements $f_1, \dots, f_{\kmax}$ fixing $p$ which  are a minimal generating set for a free Abelian group $F_{\max} \leq \Gamma$ lying in the stabliser of $p$.  Switch  to the upper half-space model $\mathbb{H}^{d+1}$ and assume that $p=\infty$, which we may do by conjugation which does not alter any dimensions.  Therefore $f_i$ acts on the boundary $\mathbb{R}^d$ by $f_i(z) = z+ t_i$ for some translation $t_i \in \mathbb{R}^d \setminus\{0\}$.  Observe that the $t_i$ must be a linearly independent set or the $f_i$ cannot be a minimal generating set for $F_{\max}$.  Let $z \in L(\Gamma) \setminus \{\infty\}$ which we know exists since $\Gamma$ is non-elementary.  By $\Gamma$-invariance of $\ls$ we have
\[
\ls \supset \Gamma(z) \supset F_{\max}(z) =  \left\{ z+\sum_{i=1}^{\kmax} n_i t_i \ : \ (n_1, \dots, n_{\kmax}) \in \mathbb{Z}^{\kmax} \right\}.
\]
Let $\alpha: \mathbb{R}^d \to \mathbb{R}^{\kmax}$ be an affine map which first sends $\{t_1, \dots, t_{\kmax} \}$ to the standard basis in $\mathbb{R}^{\kmax}$ and then translates the image of $z$ to the origin.  This is a bi-Lipschitz map and so again does not alter any dimensions.  It follow that
\[
\ad \ls \geq \ad \alpha(F_{\max}(z)) = \ad \left(\mathbb{Z}^{\kmax} \right)= \kmax.
\]
To see why $\ad \left(\mathbb{Z}^{\kmax}\right) = \kmax$ consider balls $B(0, R)$ with $R$ tending to $\infty$ and choose $r=1/2$.  Then 
\[
N_r \left(B(0,R) \cap \mathbb{Z}^{\kmax} \right)  = \# B(0,R) \cap \mathbb{Z}^{\kmax} \gtrsim R^{\kmax} \gtrsim  (R/r)^{\kmax}
\]
which proves $\ad \left(\mathbb{Z}^{\kmax} \right)\geq  \kmax$.  The other direction is trivial.

\subsubsection{Squeezing and counting horoballs}

In this section we provide some auxiliary lemmas involving horoballs.  Given a horoball $H_p$ and $\theta \in (0,1]$, we write $\theta H_p \subseteq H_p$ to denote the \emph{squeezed horoball}, which still has base point $p$ but is scaled by a factor of $\theta$.  We write $|H_p|$ to denote the Euclidean diameter of $H_p$, and thus $|\theta H_p| = \theta |H_p|$.  We also write $\Pi: \mathbb{D}^{d+1} \setminus \{\mathbf{0}\} \to \mathbb{S}^d$ to be the projection defined by choosing $\Pi(z) \in \mathbb{S}^d$ such that $\mathbf{0}, z, \Pi(z)$ are collinear with $z$ lying in-between $0$ and $\Pi(z)$.  Thus $\Pi(A)$ is the `shadow at infinity' of a set $A \subseteq \mathbb{D}^{d+1} \setminus \{ \mathbf{0}\}$.  Note that $\Pi(H_p)$ is a ball with Euclidean radius $\approx |H_p|$.   The following is a well-known result of Stratmann and Velani \cite[Corollary 3.5]{stratmannvelani}.

\begin{lma}[Corollary 3.5, \cite{stratmannvelani}] \label{squeeze}
Let $H_p$ be a standard horoball with base point $p \in P$ and $\theta \in (0,1]$ be a `squeezing factor'.  Then
\[
 \ps (\Pi(\theta H_p)) \approx  \theta^{2 \delta(\Gamma) - k(p)} |H_p|^{\delta(\Gamma)}.
\]
\end{lma}

We will also need to be able to count horoballs.  This is a standard technique in the study of Kleinian groups, see, for example, \cite{stratmannvelani}.  The following should be interpreted as a  partial localisation of  \cite[Theorem 1 and 3]{stratmannvelani}.

\begin{lma} \label{countinghoroballs}
Let $z \in \ls$ and $T> t>0$.  For $t$ sufficiently large we have
\[
 \sum_{\substack{p \in P \cap B(z,e^{-t}): \\ e^{-t}>|H_p| \geq e^{-T} }}    |H_p|^{\delta(\Gamma)} \ \lesssim \ (T-t) \ \ps(B(z,e^{-t})).
\]
\end{lma}

\begin{proof}
 It follows from the well-known `Dirichlet type Theorem' for Kleinian groups, see \cite[Theorem 1]{stratmannvelani}, that there is a constant $\kappa>0$ depending only on $\Gamma$ such that for sufficiently large $s>0$ we have
\[
\ls \subseteq \bigcup_{\substack{p \in P : \\ |H_p| \geq e^{-s} }}    \Pi\left(  \kappa \sqrt{\frac{e^{-s}}{|H_p|}} H_p \right)
\]
with multiplicity $\lesssim 1$.  In particular, for all $s>t>0$ with $t$ sufficiently large, the union 
\[
 \bigcup_{\substack{p \in P \cap B(z,e^{-t}) : \\ e^{-t}>|H_p| \geq e^{-s} }}    \Pi\left(  \kappa \sqrt{\frac{e^{-s}}{|H_p|}} H_p \right)
\]
has multiplicity $\lesssim 1$ and is contained in $B(z,2e^{-t})$.  Therefore, applying Lemma \ref{squeeze}, we have
\begin{eqnarray}
\ps(B(z,2e^{-t})) & \gtrsim  & \sum_{\substack{p \in P \cap B(z,e^{-t}): \\ e^{-t}>|H_p| \geq e^{-s} }}   \ps \left(\Pi\left(  \kappa \sqrt{\frac{e^{-s}}{|H_p|}} H_p \right) \right) \nonumber \\ \nonumber \\
& \gtrsim  & \sum_{\substack{p \in P \cap B(z,e^{-t}): \\ e^{-t}>|H_p| \geq e^{-s} }}   \left( \kappa \sqrt{\frac{e^{-s}}{|H_p|}} \right)^{2 \delta(\Gamma) - k(p)}| H_p  |^{\delta(\Gamma)}  \nonumber \\  \nonumber \\
& \gtrsim  &  e^{-s \delta(\Gamma)} \sum_{\substack{p \in P \cap B(z,e^{-t}): \\ e^{-t}>|H_p| \geq e^{-s} }}    \left(\frac{|H_p|}{e^{-s}} \right)^{k(p)/2} \nonumber \\ \nonumber \\
& \geq  &  e^{-s \delta(\Gamma)} \sum_{\substack{p \in P \cap B(z,e^{-t}): \\ e^{-t}>|H_p| \geq e^{-s} }}    1. \label{sumwith1}
\end{eqnarray}
Therefore
\begin{eqnarray*}
\sum_{\substack{p \in P \cap B(z,e^{-t}): \\ e^{-t}> |H_p| \geq e^{-T} }}    |H_p|^{\delta(\Gamma)}& \leq & \sum_{m \in \mathbb{Z} \cap [t,T+1]} \sum_{\substack{p \in P \cap B(z,e^{-t}): \\ e^{-(m-1)} >|H_p| \geq e^{-m} }}    |H_p|^{\delta(\Gamma)} \\ \\
& \lesssim & \sum_{m \in \mathbb{Z} \cap [t,T+1]} \sum_{\substack{p \in P \cap B(z,e^{-t}): \\ e^{-(m-1)} >|H_p| \geq e^{-m} }}   e^{-m\delta(\Gamma)} \\ \\
& \lesssim & \sum_{m \in \mathbb{Z} \cap [t,T+1]} e^{-m\delta(\Gamma)} \sum_{\substack{p \in P \cap B(z,e^{-t}): \\ e^{-t} >|H_p| \geq e^{-m} }}  1 \\ \\
& \lesssim & \sum_{m \in \mathbb{Z} \cap [t,T+1]} e^{-m\delta(\Gamma)} \left(e^{m \delta(\Gamma)} \ps(B(z,2e^{-t})) \right)  \qquad \text{by (\ref{sumwith1})}\\ \\
&\lesssim &(T-t) \ \ps(B(z,e^{-t}))
\end{eqnarray*}
since $\ps$ is doubling, which completes the proof.
\end{proof}

\subsubsection{Upper bound: $\ad \ls \leq  \max\{\kmax, \, \delta(\Gamma)\}$}

Recall that $\ad \ls \leq \urd \ps = \max\{ \kmax , \, 2 \delta(\Gamma) - \kmin\}$ and therefore if $\delta(\Gamma) \leq (\kmin+\kmax)/2$, then the desired upper bound $\ad \ls \leq \kmax$ follows immediately.  From now on we assume $\delta(\Gamma) > (\kmin+\kmax)/2$, although the proof we give actually works without change in the larger range $\delta(\Gamma) \geq \kmin$.  The broad strategy of our argument takes inspiration from the paper of Stratmann and Urba\'nski \cite{stratmannurbanski}.

Let $z \in \ls$, $\epsilon>0$ and $T>t>0$ with $T-t \geq \max\{\eps^{-1}, \log 10\}$.  Let $\{x_i\}_{i \in X}$ be a centred $e^{-T}$-packing of $B(z,e^{-t}) \cap \ls$ of maximal cardinality, that is, $x_i \in B(z,e^{-t}) \cap \ls$ for all $i \in X$ and $\|x_i - x_j \| > 2 e^{-T}$ for all $i \neq j$.  Decompose  $X$ as the union 
\[
X \ = X_0 \, \cup \, X_1 \, \cup \, \bigcup_{n=2}^\infty X_n
\]
where
\[
X_0 = \left\{ i \in X \ : \ (x_i)_T \in H_p \text{ with } |H_p| \geq 10 e^{-t}   \right\},
\]
\[
X_1 = \left\{ i \in X\setminus X_0 \ : \ \rho(x_i, T)  \leq \epsilon (T-t) \right\}  
\]
and
\[
X_n = \left\{ i \in X\setminus (X_0 \cup X_1) \ : \ n-1 <  \rho(x_i, T)  \leq  n \right\}.
\]
Note that this is indeed a decomposition because if $i \notin X_1$, then $\rho(x_i, T) > \epsilon (T-t) \geq 1$ (by assumption) and so $i \in X_n$ for some $n \geq 2$.

We will estimate the cardinalities  of $X_0$, $X_1$ and $X_n$ ($n \geq 2$) separately, beginning with $X_0$.   An elementary  Euclidean volume argument shows that there is at most one $p \in P$ such that $|H_p| \geq 10 e^{-t}$ and $H_p \cap \left(\cup_{i \in X} (x_i)_T \right) \neq \emptyset$.  (10 is clearly not the optimal constant here, but there is no need to optimise it.) Suppose $|X_0| \neq 0$.  It follows that we can fix  $p \in P$ with $|H_p| \geq 10 e^{-t}$ such that $(x_i)_T \in H_p$ for all $i \in X_0$.  Moreover, this forces  $z_t \in H_p$. If $\delta(\Gamma) \leq k(p)$, then by (\ref{global3})
\begin{eqnarray*}
 (e^{-t})^{\delta(\Gamma)} \exp(-\rho(z,t)(\delta(\Gamma)-k(p))) & \gtrsim & \ps(B(z,e^{-t})) \\ \\
& \gtrsim & \ps\left( \cup_{i \in X_0} B(x_i,e^{-T})\right) \\ \\
&\geq & |X_0|   \min_{i \in X_0} (e^{-T})^{\delta(\Gamma)} \exp(-\rho(x_i,T) (\delta(\Gamma)-k(p))) 
\end{eqnarray*}
since the balls $\{B(x_i,e^{-T})\}_{i \in X_0}$ are pairwise disjoint.  Therefore
\begin{eqnarray*}
|X_0| &\lesssim &  \left( \frac{e^{-t}}{e^{-T}} \right)^{\delta(\Gamma)} \max_{i \in X_0} \exp(( \rho(z,t)-\rho(x_i,T)(k(p)-\delta(\Gamma))) \\ \\
&\lesssim &  \left( \frac{e^{-t}}{e^{-T}} \right)^{\delta(\Gamma)}  \exp(( T-t)(k(p)-\delta(\Gamma))) \qquad \text{by Lemma \ref{escape2}}\\ \\
&= &  \left( \frac{e^{-t}}{e^{-T}} \right)^{k(p)} .
\end{eqnarray*}
If $\delta(\Gamma) > k(p)$, then we have to work a little harder.  In this case, decompose $X_0$  as
\[
X_0 \ = X_0^0 \, \cup \,  \bigcup_{n=1}^\infty X_0^n
\]
where
\[
X_0^0 = \left\{ i \in X_0 \ : \ \rho(x_i,T) \leq \rho(z,t)  \right\},
\]
and
\[
X_0^n = \left\{ i \in X_0 \ : \ \rho(z,t) +n-1< \rho(x_i,T)  \leq  \rho(z,t) +n \right\} .
\]
Applying (\ref{global3}) we have
\begin{eqnarray*}
 (e^{-t})^{\delta(\Gamma)} \exp(-\rho(z,t)(\delta(\Gamma)-k(p))) & \gtrsim & \ps(B(z,e^{-t})) \\ \\
& \gtrsim & \ps\left( \cup_{i \in X_0^0} B(x_i,e^{-T})\right) \\ \\
&\gtrsim& |X_0^0|    (e^{-T})^{\delta(\Gamma)} \exp(-\rho(z,t) (\delta(\Gamma)-k(p))) 
\end{eqnarray*}
and therefore
\[
|X_0^0| \lesssim  \left( \frac{e^{-t}}{e^{-T}} \right)^{\delta(\Gamma)} .
\]
If $i \in X_0^n$ for some $n \geq 1$, then the ball $B(x_i,e^{-T})$ is contained in the shadow at infinity of the squeezed horoball
\[
2e^{-(\rho(z,t)+n-1)} H_p
\]
and therefore
\begin{eqnarray*}
\ps \left( \bigcup_{i \in X_0^n} B(x_i, e^{-T}) \right) &\leq& \ps   \left( \Pi(2e^{-(\rho(z,t)+n-1)} H_p)  \right) \\ \\
 &\lesssim & e^{-(\rho(z,t)+n)(2\delta(\Gamma)-k(p))} |H_p|^{\delta(\Gamma)} \qquad \text{by Lemma \ref{squeeze}.}
\end{eqnarray*}
In the other direction, using the fact that $\{x_i\}_{i \in X_0^n}$ is an $e^{-T}$-packing,
\begin{eqnarray*}
\ps \left( \bigcup_{i \in X_0^n} B(x_i, e^{-T}) \right) &\geq & \sum_{i \in X_0^n }   \ps( B(x_i, e^{-T})) \\ \\
 &\gtrsim  &  |X_0^n|    (e^{-T})^{\delta(\Gamma)} \exp(-(\rho(z,t)+n) (\delta(\Gamma)-k(p))) 
\end{eqnarray*}
where the last line comes from (\ref{global3}) and the definition of $X_0^n$. Therefore
\[
|X_0^n| \ \lesssim \ \left(\frac{e^{-\rho(z,t)}|H_p|}{e^{-T}}\right)^{\delta(\Gamma)} e^{-n \delta(\Gamma)} \ \lesssim \ \left(\frac{e^{-t}}{e^{-T}}\right)^{\delta(\Gamma)} e^{-n \delta(\Gamma)}
\]
where the last inequality uses the  estimate
\begin{equation} \label{goodsize}
e^{-\rho(z,t)}|H_p| \approx e^{-t}.
\end{equation}
It is always true that $e^{-\rho(z,t)}|H_p| \geq  e^{-t}$ since $z_t$ is on the boundary of $e^{-\rho(z,t)}H_p$, but the reverse may not be true in general.  However, if $e^{-\rho(z,t)}|H_p| > 10 e^{-t}$, say, then the squeezed horoball $e^{-\rho(z,t)}H_p$ cannot contain any of the points $\{(x_i)_T\}_{i \in X_0}$, see Figure \ref{pic33}, and therefore $\rho(x_i,T) \leq \rho(z,t)$ for all $i \in X_0$ which renders $X_0^n$ empty for all $n \geq 1$.
\begin{figure}[H]
  \centering
  \includegraphics[width= 0.8\textwidth]{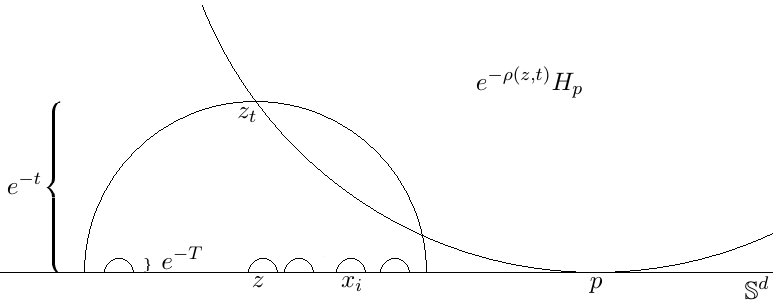}
\caption{If $e^{-\rho(z,t)}|H_p| $ is much bigger than $e^{-t}$, then it cannot intersect $\{(x_i)_T\}_{i \in X_0}$.}
\label{pic33}
\end{figure}
 We have
\begin{eqnarray*}
|X_0| & = &  |X_0^0| +  \sum_{n=1}^{\infty} |X_0^n|  \\ \\
& \lesssim &  \left( \frac{e^{-t}}{e^{-T}} \right)^{\delta(\Gamma)}  + \sum_{n=1}^{\infty} \left(\frac{e^{-t}}{e^{-T}}\right)^{\delta(\Gamma)} e^{-n \delta(\Gamma)} \\ \\
&\lesssim & \left( \frac{e^{-t}}{e^{-T}} \right)^{\delta(\Gamma)}.
\end{eqnarray*}
Therefore, irrespective of the relationship between $k(p)$ and $\delta(\Gamma)$, we have the estimate
\begin{equation} \label{X0}
|X_0| \ \lesssim  \ \left( \frac{e^{-t}}{e^{-T}} \right)^{\max\{k(p), \,  \delta(\Gamma)\}} \  \leq  \ \left( \frac{e^{-t}}{e^{-T}} \right)^{\max\{\kmax, \,  \delta(\Gamma)\}}.
\end{equation}
If $z_t \in H_p$ with $|H_p| \leq 10 e^{-t}$, then
\[
\rho(z,t) \leq \hm(z_t, z_{t-\log 10}) = \log 10 \leq T-t
\]
(by assumption) and if $z_t \in H_p$ with $|H_p| > 10 e^{-t}$ then either  $X \setminus X_0=\emptyset$ or there must be some $i \in X \setminus X_0$ such that $x_i$ is not in $H_p$.  Then we can apply Lemma \ref{escape2} to obtain
\begin{equation} \label{rhozt}
\rho(z,t) \leq  T-t + 10.
\end{equation}
Therefore we may assume the estimate (\ref{rhozt}) when estimating the size of $X \setminus X_0$.

Turning our attention to $X_1$,  using (\ref{global3}), we have
\begin{eqnarray*}
(e^{-t})^{\delta(\Gamma)} \exp(-\rho(z,t)(\delta(\Gamma)-k(z,t))) & \gtrsim & \ps(B(z,e^{-t}))  \\ \\
& \gtrsim & \ps\left( \cup_{i \in X_1} B(x_i,e^{-T})\right) \\ \\
& \gtrsim & \sum_{i \in X_1}  (e^{-T})^{\delta(\Gamma)} \exp(-\rho(x_i,T) (\delta(\Gamma)-k(x_i,T))) \\ \\
& \geq & |X_1|   (e^{-T})^{\delta(\Gamma)} \exp(-\eps(T-t) (\delta(\Gamma)-\kmin)) 
\end{eqnarray*}
where the last estimate uses the definition of $X_1$ and our assumption that $\delta(\Gamma) \geq \kmin$.  Therefore, applying (\ref{rhozt}),
\begin{eqnarray}
 |X_1|  & \lesssim &   \left( \frac{e^{-t}}{e^{-T}} \right)^{\delta(\Gamma)} e^{\eps(T-t)(\delta(\Gamma)-\kmin) } e^{(T-t) \max\{k(z,t)-\delta(\Gamma), \, 0\}}  \nonumber \\  \nonumber\\
& = & \left( \frac{e^{-t}}{e^{-T}} \right)^{\max\{ k(z,t) , \, \delta(\Gamma)\}+\eps(\delta(\Gamma) - \kmin)} \nonumber \\  \nonumber \\
& \leq & \left( \frac{e^{-t}}{e^{-T}} \right)^{\max\{ \kmax, \, \delta(\Gamma)\}+\eps(\delta(\Gamma) - \kmin)}. \label{X1}
\end{eqnarray}

Finally, we consider the sets $X_n$.  If $i \in X_n$ for $n \geq 2$, then $\rho(x_i,T) > n-1$ and $(x_i)_T \in H_p$ for some $p$ with $10 e^{-t}> |H_p| \geq e^{-T}$ and, moreover, the ball $B(x_i,e^{-T})$ is contained in the shadow at infinity of the squeezed horoball
\[
2e^{-(n-1)} H_p.
\]
Since $ |H_p|< 10e^{-t} $ we also know that $p \in B(z,10e^{-t})$.  For integer $k \in [\kmin, \kmax]$ let
\[
X_n^k  \ = \ \left\{ i \in  X_n \ : \ k(x_i,T) = k \right\}.
\]
For each set  $X_n^k$ we have
\begin{eqnarray*}
\ps \left( \bigcup_{i \in X_n^k} B(x_i, e^{-T}) \right) &\leq& \ps \left(\bigcup_{\substack{p \in P \cap B(z,10e^{-t}): \\ 10e^{-t}> |H_p| \geq e^{-T}, \\ k(p) = k}}  \Pi(2e^{-(n-1)} H_p) \right) \\ \\
&\leq& \sum_{\substack{p \in P \cap B(z,10e^{-t}): \\ 10 e^{-t}>|H_p| \geq e^{-T}  , \\ k(p) = k}}  \ps (\Pi(2e^{-(n-1)} H_p)) \\ \\
 &\lesssim & e^{-n(2\delta(\Gamma)-k)} \sum_{\substack{p \in P \cap B(z,10e^{-t}): \\ 10 e^{-t}> |H_p| \geq e^{-T}}}    |H_p|^{\delta(\Gamma)} \qquad \text{by Lemma \ref{squeeze}}  \\ \\
 &\lesssim& e^{-n(2\delta(\Gamma)-k)} (T-t+\log 10)  \ps(B(z,e^{-t})) \qquad \text{by Lemma \ref{countinghoroballs}} \\ \\
 &\lesssim& e^{-n(2\delta(\Gamma)-k)} (T-t) e^{-t \delta(\Gamma)} \exp(\rho(z,t)(k(z,t)-\delta(\Gamma)))  \quad \text{by (\ref{global3})} \\ \\
&\lesssim& e^{-n(2\delta(\Gamma)-k)} \eps^{-1} n e^{-t \delta(\Gamma)}  e^{(T-t)\max\{k(z,t)-\delta(\Gamma), \, 0\}}
\end{eqnarray*}
by applying (\ref{rhozt}).  In the last line we also used the estimate $(T-t) \leq \eps^{-1} n$, which holds provided $X_n \neq \emptyset$ (which we may assume since we are trying to bound $|X_n|$ from above).   In the other direction, using the fact that $\{x_i\}_{i \in X_n^k}$ is an $e^{-T}$-packing,
\begin{eqnarray*}
\ps \left( \bigcup_{i \in X_n^k} B(x_i, e^{-T}) \right) &\geq & \sum_{i \in X_n^k }   \ps( B(x_i, e^{-T})) \\ \\
 &\gtrsim  &  |X_n^k|    (e^{-T})^{\delta(\Gamma)} \exp(-n (\delta(\Gamma)-k)).
\end{eqnarray*}
Therefore
\[
|X_n^k| \ \lesssim \ \eps^{-1}n e^{-n \delta(\Gamma)} e^{(T-t) \delta(\Gamma)} e^{(T-t)\max\{k(z,t)-\delta(\Gamma), \, 0\}} \  \leq  \ \eps^{-1}n e^{-n \delta(\Gamma)} \left( \frac{e^{-t}}{e^{-T}} \right)^{\max\{\kmax, \,  \delta(\Gamma)\}}
\]
and it follows that
\begin{equation} \label{Xn}
|X_n|  \ \leq  \  \sum_{k =\kmin}^{\kmax} |X_n^k | \ \lesssim \  \eps^{-1}n e^{-n \delta(\Gamma)} \left( \frac{e^{-t}}{e^{-T}} \right)^{\max\{\kmax, \,  \delta(\Gamma)\}}
\end{equation}
Finally, combining (\ref{X0}), (\ref{X1}), and (\ref{Xn}), we have
\begin{eqnarray*}
|X| & = &  |X_0| \ + \  |X_1|  \ + \  \sum_{n=2}^{\infty} |X_n|  \\ \\
& \lesssim &  \left( \frac{e^{-t}}{e^{-T}} \right)^{\max\{ \kmax, \, \delta(\Gamma)\}+\eps(\delta(\Gamma) - \kmin)}  \ + \  \sum_{n=2}^{\infty} \eps^{-1}n e^{-n \delta(\Gamma)} \left( \frac{e^{-t}}{e^{-T}} \right)^{\max\{\kmax, \,  \delta(\Gamma)\}}  \\ \\
& \lesssim &  \left( \frac{e^{-t}}{e^{-T}} \right)^{\max\{ \kmax, \, \delta(\Gamma)\}+\eps(\delta(\Gamma) - \kmin)}  \ + \  \eps^{-1} \left( \frac{e^{-t}}{e^{-T}} \right)^{\max\{\kmax, \,  \delta(\Gamma)\}}
\end{eqnarray*}
which proves that $\ad \ls \leq  \max\{\kmax, \, \delta(\Gamma)\}+\eps(\delta(\Gamma) - \kmin)$ and letting $\eps \to 0$ provides the desired upper bound.

\subsection{The lower dimension of $\ls$}

\subsubsection{Upper bound: $\ld \ls \leq  \min\{\kmin, \, \delta(\Gamma)\}$}

The upper bound closely follows the lower bound in the Assouad dimension case, although we rely on a deep result of Bowditch \cite{bowditch} which we did not require an analogue of in the Assouad case. Since $\ld \ls \leq \hd \ls = \delta(\Gamma)$ it suffices to prove that $\ld \ls \leq \kmin$.  Switch to the model $\mathbb{H}^{d+1}$ and let $p \in \ls$ be a parabolic fixed point of minimal rank and assume that $p=\infty$, which we may do by conjugation which does not alter any dimensions.  A well-known classification of geometric finiteness is that every point in the limit set must either be a conical limit point or a (bounded) parabolic fixed point (but never both simultaneously).  See \cite{kapovich} for a thorough discussion of this important result which was first proved in dimension 3 by Beardon and Maskit \cite{beardonmaskit}, see also \cite{bishop}, and in higher dimensions it is due to Bowditch \cite[Definition (GF2)]{bowditch}.  In particular, $p=\infty$ is not a conical limit point. Applying the definition of bounded parabolic point to $p=\infty$, see \cite[Definition, page 272]{bowditch}, one obtains the following lemma.   We are grateful to John Parker for bringing this fact to our attention. 
\begin{lma}\label{bowlemma}
There exists $\lambda >0$ and a $\kmin$-dimensional linear  subspace $V \subseteq \mathbb{R}^d$ such that $\ls \subseteq V _\lambda \cup \{\infty\}$ where  $V_\lambda = \{ x \in \mathbb{R}^d \ : \ \inf_{y \in V} \|x-y\| \leq \lambda\}$ denotes the Euclidean $\lambda$-neighbourhood of $V$.
\end{lma}
 Interestingly, this lemma relies on geometric finiteness, see \cite{parkerstratmann}, but our proof of the lower bound in the Assouad dimension case is valid for \emph{any} non-elementary Kleinian group.  Let $\infty \neq z \in \ls$ and consider balls $B(z, R)$ with $R$ tending to $\infty$ and choose $r=2\lambda$.  Then, by Lemma \ref{bowlemma}
\[
N_r \left(B(z,R) \cap \ls \right)  \leq N_r \left(B(z,R) \cap V_\lambda \right)  \lesssim  (R/r)^{\kmin}
\]
which proves $\ld \ls \leq  \kmin$, as required.

\subsubsection{Lower bound: $\ld \ls \geq  \min\{\kmin, \, \delta(\Gamma)\}$}

The lower bound is philosophically similar to the upper bound in the Assouad dimension case, although the details turn out to be rather different. Recall that $\ld \ls \geq \lrd \ps = \min\{ \kmin , \, 2 \delta(\Gamma) - \kmax\}$ and therefore if $\delta(\Gamma) \geq (\kmax+\kmin)/2$, then the desired lower bound $\ld \ls \geq \kmin$ follows immediately and therefore we assume from now on that $\delta(\Gamma) < (\kmax+\kmin)/2 \leq \kmax$.

Let $z \in \ls$, $\epsilon \in (0,1)$ and $T>t>1$ with $T-t \geq \min\{\eps^{-1}, \, \log 10\}$.  Let $\{B(y_i, e^{-T})\}_{i \in Y}$ be a centred $e^{-T}$-cover of $B(z,e^{-t}) \cap \ls$ of minimal cardinality.  By `centred' we mean that $y_i \in B(z,e^{-t}) \cap \ls$ for all  $i \in Y$.   Decompose  $Y$ as the union 
\[
Y \ = Y_0 \, \cup \, Y_1 
\]
where
\[
Y_0 = \left\{ i \in Y \ : \ (y_i)_T \in H_p \text{ with } |H_p| \geq 10 e^{-t}   \right\}
\]
and
\[
Y_1 = Y \setminus Y_0.
\]
 Since $\{B(y_i, e^{-T})\}_{i \in Y}$ is a cover of $B(z,e^{-t}) \cap \ls$ we have
\begin{eqnarray} \label{pigeon}
 \ps(B(z,e^{-t})) &\leq&  \ps\left( \cup_{i \in Y} B(y_i,e^{-T})\right) \nonumber \\ \nonumber \\
&  \leq& \ps\left( \cup_{i \in Y_0} B(y_i,e^{-T})\right) \ +  \ \ps\left( \cup_{i \in Y_1} B(y_i,e^{-T})\right)
\end{eqnarray}
and therefore at least one of the two  terms in (\ref{pigeon}) must be at least $ \ps(B(z,e^{-t}))/2$.  We will consider each of these possibilities separately, beginning with the term involving $Y_0$.   In this case  $Y_0 \neq \emptyset$ and therefore, as above, we know that we can fix  $p \in P$ with $|H_p| \geq 10 e^{-t}$ such that $(y_i)_T \in H_p$ for all $i \in Y_0$ and, moreover, that  $z_t \in H_p$. If $\delta(\Gamma) \geq k(p)$, then by (\ref{global3})
\begin{eqnarray*}
 (e^{-t})^{\delta(\Gamma)} \exp(-\rho(z,t)(\delta(\Gamma)-k(p))) & \lesssim & \ps(B(z,e^{-t})) \\ \\
&\leq & 2 \, \ps\left( \cup_{i \in Y_0} B(y_i,e^{-T})\right) \\ \\
&\lesssim & |Y_0|   \max_{i \in Y_0} (e^{-T})^{\delta(\Gamma)} \exp(-\rho(y_i,T) (\delta(\Gamma)-k(p))) 
\end{eqnarray*}
and therefore
\begin{eqnarray}
|Y_0| &\gtrsim &  \left( \frac{e^{-t}}{e^{-T}} \right)^{\delta(\Gamma)} \min_{i \in Y_0} \exp(( \rho(z,t)-\rho(y_i,T)(k(p)-\delta(\Gamma))) \nonumber \\  \nonumber \\
&\gtrsim &  \left( \frac{e^{-t}}{e^{-T}} \right)^{\delta(\Gamma)}  \exp(( T-t)(k(p)-\delta(\Gamma))) \qquad \text{by Lemma \ref{escape2}} \nonumber \\  \nonumber \\
&= &  \left( \frac{e^{-t}}{e^{-T}} \right)^{k(p)} \label{dk1}.
\end{eqnarray}
Now suppose  $\delta(\Gamma) < k(p)$ and write
\[
Y_0^0 = \left\{ i \in Y_0 \ : \ \rho(y_i,T) \leq  \rho(z,t)  \right\}.
\]
If $ i \in Y_0 \setminus Y_0^0$  then $\rho(y_i,T) >  \rho(z,t)$ and therefore $(y_i)_T \in e^{-\rho(z,t)} H_p$.  The Euclidean distance from $(y_i)_T$ to $\mathbb{S}^{d}$ is $ e^{-T}$ and  $e^{-\rho(z,t)} |H_p| \approx e^{-t}$ (recall that this is implied by $Y_0 \setminus Y_0^0 \neq \emptyset$, see (\ref{goodsize})).  Therefore, writing $\eta$ for the Euclidean radius of $ e^{-\rho(z,t)} H_p$,  Pythagoras' Theorem guarantees that
\[
\|y_i - p \| \leq   \sqrt{\eta^2 - (\eta-e^{-T})^2} \lesssim \sqrt{\eta e^{-T}} \lesssim \sqrt{e^{-t}e^{-T}},
\]
see Figure \ref{pic44}.
\begin{figure}[H]
  \centering
  \includegraphics[width= 0.6\textwidth]{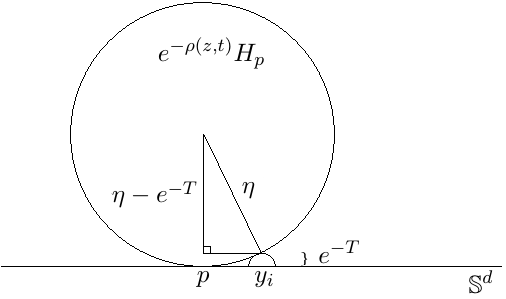}
\caption{A right-angled triangle with vertices at $(y_i)_T$,  $p_T$ and the centre of $ e^{-\rho(z,t)} H_p$.}
\label{pic44}
\end{figure}
\noindent It follows that  $B(y_i, e^{-T})$ is contained in the shadow at infinity of the squeezed horoball
\[
\kappa' \sqrt{\frac{e^{-T}}{e^{-t}}} e^{-\rho(z,t)} H_p
\]
for some $\kappa' \approx 1$.  Therefore, by Lemma \ref{squeeze}, we have
\begin{eqnarray*}
\ps \left( \cup_{i \in Y_0 \setminus Y_0^0} B(y_i, e^{-T})  \right) &\lesssim & \left(\sqrt{\frac{e^{-T}}{e^{-t}}}\right)^{2\delta(\Gamma)-k(p)} \ps \left( \Pi \left(  e^{-\rho(z,t)} H_p\right) \right)\\ \\
&\approx & e^{(t-T)(\delta(\Gamma)-k(p)/2)} \ps(B(z,e^{-t}))
\end{eqnarray*}
since $e^{-\rho(z,t)} |H_p| \approx e^{-t}$.  Since $\delta(\Gamma)-k(p)/2>0$, this proves that for $T-t$ sufficiently large, balls with centres in $Y_0 \setminus Y_0^0$ cannot carry a fixed proportion of the  $\ps(B(z,e^{-t}))$ and so
\[
\ps \left( \cup_{i \in Y_0^0} B(y_i, e^{-T})  \right) \approx \ps\left( \cup_{i \in Y_0} B(y_i,e^{-T})\right) \geq \ps(B(z,e^{-t}))/2.
\]
Therefore
\begin{eqnarray*}
 (e^{-t})^{\delta(\Gamma)} \exp(\rho(z,t)(k(p)-\delta(\Gamma))) & \lesssim & \ps(B(z,e^{-t})) \qquad \text{by (\ref{global3})} \\ \\
& \lesssim  &  \ps\left( \cup_{i \in Y_0^0} B\left(y_i,e^{-T}\right)\right) \\ \\
&\lesssim& |Y_0^0|    (e^{-T})^{\delta(\Gamma)} \exp(\rho(z,t) (k(p)-\delta(\Gamma))) 
\end{eqnarray*}
by (\ref{global3}) and the definition of $Y_0^0$.  This yields
\[
|Y_0 |  \ \geq \ |Y_0^0| \ \gtrsim   \ \left( \frac{e^{-t}}{e^{-T}} \right)^{\delta(\Gamma)}
\]
which, together with (\ref{dk1}), shows that, irrespective of the relationship between $k(p)$ and $\delta(\Gamma)$, we have the estimate
\begin{equation} \label{Y0}
|Y_0| \ \gtrsim  \ \left( \frac{e^{-t}}{e^{-T}} \right)^{\min \{k(p), \,  \delta(\Gamma)\}} \  \geq  \ \left( \frac{e^{-t}}{e^{-T}} \right)^{\min\{\kmin, \,  \delta(\Gamma)\}}.
\end{equation}

It remains to consider the case where the second term in (\ref{pigeon}) carries at least half the mass, that is
\[
\ps\left( \cup_{i \in Y_1} B(y_i,e^{-T})\right) \geq \ps(B(z,e^{-t})/2.
\]
Since this guarantees $Y_1 \neq \emptyset$, we may assume the estimate (\ref{rhozt}).  Write
\[
Y_1^0 = \left\{ i \in Y_1 \ : \ \rho(y_i,T) \leq  \eps(T-t)  \right\}.
\]
If $ i \in Y_1 \setminus Y_1^0$  then $\rho(y_i,T) >  \eps(T-t) $ and therefore $(y_i)_T \in e^{-\eps(T-t)} H_p$ for some $H_p$ with basepoint $p \in B(z, 10 e^{-t})$ satisfying $e^{-T} \leq |H_p| < 10 e^{-t}$.  Since the Euclidean distance from $(y_i)_T$ to $\mathbb{S}^{d}$ is $ e^{-T}$, we can argue as above (see Figure \ref{pic44}) using Pythagoras' Theorem to show that
\[
\|y_i - p \| \lesssim \sqrt{e^{-\eps(T-t)}|H_p|e^{-T}}
\]
and therefore $B(y_i, e^{-T})$ is contained in the shadow at infinity of the squeezed horoball
\[
\kappa'' \sqrt{\frac{e^{-\eps(T-t)-T}}{|H_p|}} H_p
\]
for some $\kappa'' \approx 1$.  Therefore
\begin{eqnarray*}
\ps \left( \cup_{i \in Y_1 \setminus Y_1^0} B\left(y_i, e^{-T}\right)  \right) &\leq& \sum_{\substack{p \in P \cap B(z,10e^{-t}): \\ 10e^{-t}>|H_p| \geq e^{-T} }} \ps \left( \Pi \left( \kappa'' \sqrt{\frac{e^{-\eps(T-t)-T}}{|H_p|}} H_p \right) \right) \\ \\  
& \approx &   \sum_{\substack{p \in P \cap B(z,10e^{-t}): \\ 10e^{-t}>|H_p| \geq e^{-T} }} \left( \sqrt{\frac{e^{-\eps(T-t)-T}}{|H_p|}} \right)^{2\delta(\Gamma) - k(p)} |H_p|^{\delta(\Gamma)}  \\ \\
&\,& \hspace{60mm} \qquad \text{by Lemma \ref{squeeze}} \\ \\  
& \leq &   e^{-\eps(T-t)(\delta(\Gamma)-\kmax/2) }  \hspace{-5mm}\sum_{\substack{p \in P \cap B(z,10e^{-t}): \\ 10 e^{-t}>|H_p| \geq e^{-T} }} \left( \sqrt{\frac{e^{-T}}{|H_p|}} \right)^{2\delta(\Gamma) - k(p)} \hspace{-2mm} |H_p|^{\delta(\Gamma)}  \\ \\ 
& \leq &    e^{-\eps(T-t)(\delta(\Gamma)-\kmax/2) }\sum_{\substack{p \in P \cap B(z,10e^{-t}): \\ 10 e^{-t}>|H_p| \geq e^{-T} }}  |H_p|^{\delta(\Gamma)}  \\ \\   
& \lesssim &     e^{-\eps(T-t)(\delta(\Gamma)-\kmax/2) }(T-t+\log 10) \ps(B(z,e^{-t})) 
\end{eqnarray*}
by Lemma \ref{countinghoroballs}.  Since $\delta(\Gamma)-\kmax/2>0$ this proves that for $T-t$ sufficiently large, balls with centres in $Y_1 \setminus Y_1^0$ cannot carry a fixed proportion of $\ps(B(z,e^{-t}))$ and so
\[
\ps \left( \cup_{i \in Y_1^0} B(y_i, e^{-T})  \right) \approx \ps\left( \cup_{i \in Y_1} B(y_i,e^{-T})\right) \geq \ps(B(z,e^{-t})/2.
\]
 It then follows from (\ref{global3}) that 
\begin{eqnarray*}
(e^{-t})^{\delta(\Gamma)} \exp(-\rho(z,t)(\delta(\Gamma)-k(z,t))) & \lesssim & \ps(B(z,e^{-t}))  \\ \\
& \lesssim & \ps\left( \cup_{i \in Y_1^0} B(y_i,e^{-T})\right) \\ \\
& \lesssim & \sum_{i \in Y_1^0}  (e^{-T})^{\delta(\Gamma)} \exp(-\rho(y_i,T) (\delta(\Gamma)-k(y_i,T)))  \\ \\
& \leq & |Y_1^0|   (e^{-T})^{\delta(\Gamma)} \exp(\eps(T-t) (\kmax-\delta(\Gamma))) 
\end{eqnarray*}
where the last estimate uses the definition of $Y_1^0$ and our assumption that $\delta(\Gamma) \leq \kmax$.  Therefore, applying (\ref{rhozt}),
\begin{eqnarray}
 |Y_1| \  \geq  \ |Y_1^0|  & \gtrsim &   \left( \frac{e^{-t}}{e^{-T}} \right)^{\delta(\Gamma)} e^{\eps(T-t)(\delta(\Gamma)-\kmax) } e^{(T-t) \min\{k(z,t)-\delta(\Gamma), \, 0\}}  \nonumber \\  \nonumber\\
& = & \left( \frac{e^{-t}}{e^{-T}} \right)^{\min\{ k(z,t) , \, \delta(\Gamma)\}+\eps(\delta(\Gamma) - \kmax)} \nonumber \\  \nonumber \\
& \geq & \left( \frac{e^{-t}}{e^{-T}} \right)^{\min\{ \kmin, \, \delta(\Gamma)\}+\eps(\delta(\Gamma) - \kmax)}. \label{Y1}
\end{eqnarray}
We have proved that at least one of (\ref{Y0}) and (\ref{Y1}) must hold and therefore
\begin{eqnarray*}
|Y| \ = \  |Y_0|+|Y_1| \  \gtrsim  \   \left( \frac{e^{-t}}{e^{-T}} \right)^{\min\{ \kmin, \, \delta(\Gamma)\}+\eps(\delta(\Gamma) - \kmax)} 
\end{eqnarray*}
which proves that $\ld \ls \geq  \min\{\kmin, \, \delta(\Gamma)\}-\eps(\kmax - \delta(\Gamma) )$ and letting $\eps \to 0$ provides the desired lower bound.

\section{Local dimensions of $\ps$: proof of Proposition \ref{localdims}} \label{localdimsproof}

Let $z \in \ls$ and $t>0$.  Then combining (\ref{global3}) and the fact that $\rho(z,t) \leq t$ gives
\begin{eqnarray*}
\uld (\ps, z) =  \delta(\Gamma)+  \limsup_{t \to \infty} \frac{\rho(z,t)(\delta(\Gamma)-k(z,t))}{t} & \leq& \delta(\Gamma) + \max\{0, \delta(\Gamma) - \kmin\} \\
& =& \max\left\{\delta(\Gamma), \,  2\delta(\Gamma) - \kmin\right\}
\end{eqnarray*}
and
\begin{eqnarray*}
\lld (\ps, z) =  \delta(\Gamma)+  \liminf_{t \to \infty} \frac{\rho(z,t)(\delta(\Gamma)-k(z,t))}{t} & \geq& \delta(\Gamma) + \min\{0, \delta(\Gamma) - \kmax\} \\
& =& \min\left\{\delta(\Gamma), \,  2\delta(\Gamma) - \kmax\right\}.
\end{eqnarray*}
Moreover, the local dimension $\delta(\Gamma)$ is achieved at $\ps$-typical $z$, see \cite{stratmannvelani}, and if $p$ a parabolic fixed point of rank $k(p)$ associated with a standard horoball $H_p$, then the above estimates combined with Lemma \ref{parapoint} yield
\[
\dim_{\textup{loc}} (\ps, p) \ = \ 2\delta(\Gamma) - k(p)
\]
which completes the proof.

\vspace{1mm}

\begin{samepage}

\subsection*{Acknowledgements}

The  author was   supported by a \emph{Leverhulme Trust Research Fellowship} (RF-2016-500).  Much of the work was completed while the author was resident at the  Institut Mittag-Leffler during the 2017  semester programme \emph{Fractal Geometry and Dynamics} and  is very grateful for the inspiring atmosphere he found there.  He is extremely grateful to John Parker for many helpful conversations about Kleinian groups and also thanks Douglas Howroyd for stimulating discussions regarding the regularity dimensions of measures.

Finally, the author wishes to acknowledge Bernd Stratmann, whose tragic passing in 2015 was a terrible loss to the community.  The author first became interested in hyperbolic geometry through Bernd and is grateful for many interesting conversations over the years.

\end{samepage}

\vspace{10mm}

\begin{samepage}

\noindent \emph{Jonathan M. Fraser\\
School of Mathematics and Statistics\\
The University of St Andrews\\
St Andrews, KY16 9SS, Scotland} \\
\noindent  Email: jmf32@st-andrews.ac.uk\\ \\

\end{samepage}

\end{document}